\documentclass[preprint,12pt]{elsarticle}

\biboptions{numbers,sort&compress}
\usepackage{amssymb}
\usepackage{amsmath}
\usepackage{amsthm}
\usepackage{mathtools}
\usepackage{graphicx}

\journal{Journal of Mathematical Psychology}

\newcommand{\Z}{{\mathbb Z}}

\newcommand{\lt}{\left}
\newcommand{\rt}{\right}

\newcommand{\wt}{\widetilde}
\newcommand{\eps}{\varepsilon}

\newtheorem{theorem}{Theorem}
\newtheorem{lemma}{Lemma}
\newtheorem{prop}{Proposition}
\newtheorem{cor}{Corollary}

\newcommand{\appendixpage}{}

\begin{document}

\begin{frontmatter}

\title{Classification of dynamics for a two person model of planned behavior}

\author[label1]{Rishi Dadlani\corref{cor1}}
\ead{rdadlani@sas.upenn.edu}
\cortext[cor1]{Corresponding Author}

\author[label2,label3]{John S. McAlister}

\author[label4]{Tahra L. Eissa}

\author[label3,label5,label6]{Nina H. Fefferman}

\affiliation[label1]{organization={University of Pennsylvania, Department of Mathematics}}

\affiliation[label2]{organization={University of Tennessee - Knoxville, Department of Mathematics}}

\affiliation[label3]{organization={National Institute for Modeling Biological Systems}}

\affiliation[label4]{organization={University of Colorado Anschutz school of medicine}}

\affiliation[label5]{organization={US NSF Center for Analysis and Prediction of Pandemic Expansion (APPEX)}}

\affiliation[label6]{organization={School of Mathematical and Natural Sciences, Arizona State University}}

\begin{abstract}
We study a dynamical system modeling the Theory of Planned Behavior (TPB) in which each individual’s behavioral intention $x_i(t)\in(-1,1)$ evolves continuously under an ODE driven by internal attitude, perceived social norms, and perceived behavioral control, while \emph{actions} occur as discrete threshold events: when $x_i$ reaches a fixed threshold it is reset to $0$ and produces a transient “nudge” that jumps to $1$ and then decays exponentially. This yields a hybrid ODE–threshold system with psychologically interpretable parameters. We derive a partial classification in the general case of $n$ individuals.

Focusing on the two-individual case ($n=2$), we obtain explicit formulas for trajectories between action events and derive bounds for first-action times. In the mixed setting where one individual is intrinsically increasing  and the other is not, we identify a scalar invariant $M$ measuring the net effect of one period of excitation. We prove that $M\le 0$ is \emph{equivalent} to a partial-action state (only the intrinsically active individual acts countably infinitely often), while $M>0$ is \emph{equivalent} to full action (both individuals act countably infinitely often). Finally, we demonstrate numerically that these analytic boundaries partition the $(\alpha_1,\alpha_2)$ parameter space with near-perfect agreement, and we provide exploratory simulations suggesting analogous structures for three individuals.
\end{abstract}


\end{frontmatter}

\section{Introduction}\label{sec:introduction}
Predicting the variation in and interactions among individuals remains the fundamental challenge in understanding collective action as an emergent complex system. Even when individuals face nearly identical external cues, their choices can diverge dramatically and, through nonlinear feedbacks, these micro-level divergences can reshape the macroscopic trajectory of an entire population. In complex systems, where emergent patterns often hinge on isolated events or the precise ordering of early decisions, the inability to anticipate individual actions is not simply a matter of noise but a structural barrier to predicting collective outcomes. This challenge is well recognized across the mathematical and physical sciences from opinion and belief dynamics \cite{degroot1974reaching, hegselmann2002opinion, nowak2019nonlinear}, to contagion-like models of social behavior \cite{granovetter1978threshold, watts2002simple, spears2021social}, to excitable media and other threshold-driven systems whose global dynamics depend sensitively on micro-level activation thresholds \cite{acharya2021neuromorphic, marsh2024emergent}.

Within psychology, the ``Theory of Planned Behavior’’ (TPB; \cite{ajzen1991theory, bosnjak2020theory}) offers a principled framework for describing when individuals choose to act: intention is shaped jointly by internal attitude, perceived social norms, and perceived behavioral control. From a dynamical-systems perspective, TPB suggests a natural analog to excitable systems \cite{ciszak2004coupling,gerstner2002spiking}. Attitudes and perceived norms behave as slow variables, accumulating or dissipating under environmental and social influence; intention is a membrane-potential-like state variable; and action occurs only when this intention crosses a threshold. The subsequent feedback of observed action into perceived norms mirrors classical ‘reset and excitation’ processes in integrate-and-fire neuron models \cite{gerstner2002spiking, desroches2022classification} and in more general excitable systems where local spikes influence the future excitability of neighbors \cite{perc2007stochastic, jirsa2004connectivity}.

Earlier work introduced a simplified dynamical version of TPB using a system of ordinary differential equations to track how individual intentions to act might evolve under these interacting components and a perceptual immediacy bias \cite{schwarze2024planned}. Actions occurred when intention exceeded a fixed threshold; upon acting, individuals generated “nudges” that transiently increased the perceived social norm for others, thereby shifting the conditions under which those individuals might themselves spike in intention to act. This scenario creates a hybrid dynamical system strikingly reminiscent of integrate-and-fire assemblies, with intention evolving continuously under ODE dynamics, punctuated by discrete threshold-crossing events that feed back into the continuous evolution of others. Numerical simulations of that model demonstrated a broad repertoire of behaviors familiar to those studying excitable systems: delayed spikes, rapid cascades, partial activation, oscillatory firing patterns, and long transients that obscure the eventual dynamical regime — even under deterministic, perfectly observed conditions.

In this paper, we take a deliberate step toward a formal mathematical understanding of this system by focusing on the two-individual case, the simplest nontrivial instantiation of the TPB-inspired ODE–threshold hybrid. Our goal is not yet to resolve the full many-body problem, but rather to establish analytic traction on the core temporal structures governing these dynamics. The two-individual system is small enough to permit explicit reasoning, but rich enough to exhibit the characteristic features of thresholded excitable dynamics: timing sensitivity, sequence dependence, and subtle interactions between continuous intention trajectories and discrete action events.

We derive bounds on the first-passage time until the initial action occurs, as well as bounds on subsequent action times and action counts under specific structural conditions. These results parallel classical analyses of spike-timing bounds in excitable models \cite{izhikevich2007dynamical, brette2015philosophy, markram2011history}, but here the excitability landscape is shaped by psychological constructs rather than varying ionic conductance. We also characterize constraints on the ``ordering of action events,’’ again echoing excitable-network phenomena in which the timing of early spikes determines the basin of attraction for the global regime.

Finally, we provide a complete classification of the asymptotic regimes of the two-individual system, showing that the dynamics fall into one of three mutually exclusive categories:
\begin{itemize}
\item ``No action’’, in which intention trajectories never cross the threshold into action;
\item ``Partial action’’, in which at least one individual fires countably infinitely many times but the other does not;
and
\item``Full action’’, where both individuals fire countably many times.
\end{itemize}
This tripartite structure reflects canonical categories in excitable networks (i.e., quiescent, partially active, and persistently active regimes), yet arises here from the interaction of attitude, perceived norms, perceived behavioral control, and perceptual bias.

Establishing these results for the two-individual case provides the first rigorous analytical foundation for TPB-based dynamical models of collective action. While our earlier numerical work highlighted the richness of the model, analytic understanding requires grappling directly with the hybrid excitable nature of the system. The two-individual case allows us to do so cleanly. These proofs and classifications open the way toward studying larger populations (we provide brief numerical simulations of the classification of three-individual systems, but leave the proof of any such dynamics to future work), heterogeneous topologies, multiple perceptual biases (e.g., primacy or saliency), stochastic excitation, and the embedding of these dynamics into networked excitable structures.

Ultimately, by formalizing the dynamical behavior of even the simplest TPB-inspired system, we move toward a rigorous mathematical theory linking internal psychological states, threshold-driven action, and the emergent properties of collective behavior. This theoretical bridge between excitable systems and human decision-making creates space for genuinely predictive modeling of social action grounded in both cognitive realism and complex-systems mathematics.

The model described in \cite{schwarze2024planned} with linear functional forms is a $2n$ hybrid system combining ODEs and difference equations in each variable with non-constant intervals of hybridization. Table \ref{tab:KeyParemetersandVariables} describes each of the parameters. 
\[
\dot x_i = [\sigma_A \alpha_i + \sigma_S(\gamma_i - \mu_S)]\sigma_C (\gamma_i + \mu_C)(1-x_i)(1+x_i)
\]
\[\dot y_i = -ry_i\]
\begin{table}[h]
    \begin{center}
    \begin{tabular}{cccccc}
        \textbf{Symbol} & \textbf{Name} & \textbf{Constraint}\\
        $n$ & Population size & $n \in \Z^+$ \\
        $x_i$ & Behavioral intention & $x_i \in (-1,1)$ \\
        $y_i$ & Nudge effect & $y_i \in [0,1]$ \\
        $\gamma_i$ & Perceived-control coefficient & $\gamma_i \in [0,1]$ \\
        $r$ & Immediacy parameter & $r \geq 0$ \\
        $\tau$ & Action-taking threshold & $\tau \in (0,1) $\\
        $\alpha_i$ & Internal attitude & $\alpha_i \in (-1,1)$ \\
        $\sigma_A$ & Strength of internal attitudes & $\sigma_A \geq 0$ \\
        $\sigma_S$ & Strength of social norms & $\sigma_S \geq 0$ \\
        $\sigma_C$ & Strength of perceived behavioral control & $\sigma_C \geq 0$ \\
        $\mu_S$ & Social norm threshold & $\mu_S \in [0,1]$ \\
        $\mu_C$ & Baseline perceived control & $\mu_ C\geq 0$ \\
    \end{tabular}
    \caption{Key Parameters and Variables}
    \label{tab:KeyParemetersandVariables}
    \end{center}
\end{table}
\[\gamma_i = \frac{1}{n-1}\sum_{j \neq i}y_j\]
That is, $\gamma_i$ is the average of the nudge effects of all but the $i$th individual.

We say individual $i$ \textit{acts} at time $t$ if $x_i(t) \geq \tau$. For infinitesimal $dt$, the system has the following property:
\[x_i(t + dt) = 0 \quad \text{ and } \quad y_i(t + dt) = 1 \]
As there are many variables, the following description in plain English may be useful.

The individual's behavioral intention will be increasing, decreasing, or remain constant depending on the sign of $\sigma_A\alpha_i - \sigma_S\mu_S$ since $\sigma_C\mu_C(1-x_i)(1+x_i)$ is always non-negative for the constraints on each parameter. The logistic term $(1-x_i)(1+x_i)$ allows for, all else being equal, a larger derivative near $0$ and a smaller derivative near $\pm 1$. In the case where an individual acts, their behavioral intention at the time of action discontinuously updates to 0, and their nudge effect discontinuously updates to 1. This induces an update to each other individual's $\gamma_i$, causing an increase in $\dot x_i$. This update exponentially decays according to $\dot y_i$. Note also that the only parameters that are dependent on the particular individual are $x_i,y_i,\gamma_i,\alpha_i$, and only $x_i(0), y_i(0)$ and $\alpha_i$ are pre-set.

Solutions to this hybrid model can be simulated through standard ODE numerical methods with mild modifications to account for the state dependent discontinuities. A simulated solution is pictured in figure \ref{fig:SimulatedSolution}.
\begin{figure}
    \centering
    \includegraphics[width=\linewidth]{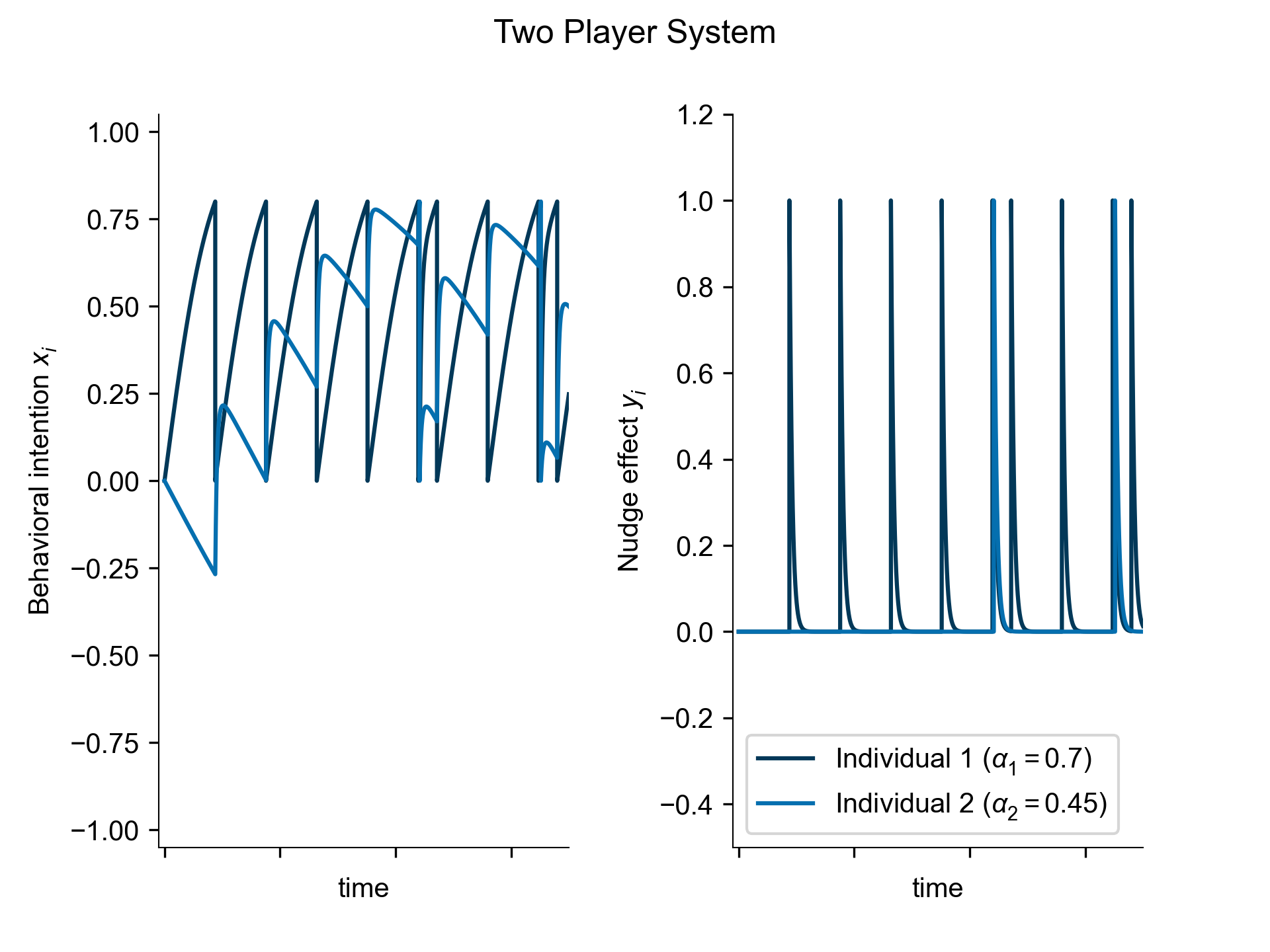}
    \caption{A simulated solution for a two individual case. Individual 1 (dark blue) has an inherently increasing behavioral intention and when this individual reaches the threshold, the behavioral intention is set back to zero, and a nudge (right panel) is created. This nudge increases individual 2's (light blue) behavioral intention momentarily, but not enough to cause individual 2 to act. After several actions by individual 1, individual 2 is finally incited to act, causing a nudge which increases individual 1's behavioral intention. This behavior will continue indefinitely }
    \label{fig:SimulatedSolution}
\end{figure}

\section{Preliminary Results}\label{sec:PreliminaryResults}

\begin{prop}\label{PROP:Cinfty}
    If no individual has acted in a particular interval, then $x_i$ is $C^\infty$ in that interval.
\end{prop}
\begin{proof} Without the discontinuous updates, the ODE governing $\dot x_i$ is composed entirely of products and sums of $C^\infty$ functions. \end{proof}
Note this implies that $x_i(t) > \tau$ is impossible unless $t= 0 $ and we took $x_i(0) > \tau$ as an initial condition. Even so, we may as well treat this as $x_i(0) = \tau$ since $x_i(dt) = 0$ in either case.

The following result allows us to follow the behavior of $x_i$ before any individual has acted. Although we rarely use this solution, it is relatively simple. In a following section (Proposition \ref{Prop:explicitSolution}), we will derive a more useful -- and complicated -- expression in a similar way for the case with only two individuals.
\begin{prop}\label{PROP:beforeNudgeSolution}
    While $\gamma_i(t) = 0$, $x_i(0) = u_0$, $x_i$ has an explicit solution in the form
    $\tanh(A_it + \tanh^{-1}(u_0))$ where $A_i = [\sigma_A \alpha_i - \sigma_S \mu_S]\sigma_C\mu_C$.
\end{prop}
\begin{proof} While no individual has acted, the system is uncoupled, thus every individual acts autonomously. As such, and for convenience, we drop the index $i$ on $x_i$, $\alpha_i$, $A_i$. With the given conditions: \[
\frac{dx}{dt} = A(1-x)(1+x) = A(1-x^2)\] which (for the purpose of solving the separable ODE) gives us
\[\frac{dx}{1-x^2} = Adt
\]
Integrating,
\[ \int \frac{dx}{1-x^2} = At + C\]
where $C$ is a constant.
\[
\int \frac{dx}{1-x^2} = \int \lt(\frac{1}{2(x+1)}  - \frac{1}{2(x-1)} \rt)dx = \frac{1}{2}\int \frac{dx}{x+1} - \frac{1}{2}\int \frac{dx}{x-1}
\]
which gives
\[\frac{1}{2}\ln\lt(\lt|\frac{1+x}{1-x}\rt|\rt) = At + C\]
we are able to drop the absolute value bars due to the constraints on $x$. Using the initial condition $x(0) = u_0$, \[
\frac{1}{2}\ln\lt(\frac{1+u_0}{1-u_0}\rt) = \tanh^{-1}(u_0) = C
\] \[
   t  =  \frac{1}{A}( \tanh^{-1}(x) - \tanh^{-1}(u_0))
\]
Solving for $x$, we get $\tanh(At + \tanh^{-1}(u_0)) = x$, the desired result. \end{proof}

In the proposition below, recall that $\tau$ is the threshold for action.
\begin{cor}\label{period}
    Suppose that for individual $i$, $x_i(0) = u_0$, $\gamma_i(t) = 0$ for all $t \leq \frac{1}{A}( \tanh^{-1}(\tau) - \tanh^{-1}(u_0))$ and $\sigma_A\alpha_i - \sigma_S\mu_S > 0$.  Then, the time until individual $i$'s first action is $T = \frac{1}{A}( \tanh^{-1}(\tau) - \tanh^{-1}(u_0))$.
\end{cor}
\begin{proof}
Until $x$ reaches the threshold $\tau$, it is $C^\infty$, so we are able to evaluate the expression \[
t = \frac{1}{A}( \tanh^{-1}(x) - \tanh^{-1}(u_0))
\] derived in Proposition \ref{PROP:beforeNudgeSolution} on $x = \tau$ to get \[
T = \frac{1}{A}( \tanh^{-1}(\tau) - \tanh^{-1}(u_0))
\]
as desired. \end{proof}

This says that we can find how long it takes an individual to act in the absence of nudge effects given the individual has increasing intention. As in the preceding proposition, we rarely use this exact expression, just referring to the time as $T$ in the future, but it is nice knowing there is such an expression if one decided to do explicit computations or examples.

We will prove a related corollary that applies while $\gamma_i(t) \neq 0$, but first, we need the lemma below, which is a re-framing of an elementary result in analysis.
\begin{lemma}\label{LEMMA:derivativeDef}
    Suppose $x_i(t_0) = x_j(t_0)$ and $\dot x_i(t_0) > \dot x_j(t_0)$, then there is a $\delta > 0$ and an interval where $t \in (t_0, t_0 + \delta)$ satisfies $x_i(t) > x_j(t)$
\end{lemma}
\begin{proof} First, if $x_i(t_0)  \geq \tau$, then $x_i(t_0 + dt) = 0$, otherwise, $x_i(t_0 + dt) < \tau$. Note that by definition of the derivative,
\[
x(t_0 + dt) = \dot x(t_0)dt + x(t_0) + o(dt) \text{ as } dt \to 0
\]
By definition,
\[
o(dt)/dt \to 0 \text{ as } dt \to 0
\]
so set $o(dt)/dt = \eps(dt)$, so  \[
o(dt) = \eps(dt)dt \text{ as } dt \to 0
\]
Hence, \[
x_i(t_0 + dt) - x_j(t_0 + dt) = (\dot x_i(t_0) - \dot x_j(t_0)) dt + \eps(dt) dt \text{ as } dt \to 0
\]
Call $\dot x_i(0) - \dot x_j(0) = c > 0$, since $\eps(dt) \to 0$, there exists $\delta > 0$ where\[
|dt| < \delta \implies |\eps(dt)| < c/2
\]
Then, for $dt \in (0,\delta)$\[
x_i(t_0  +dt) - x_j(t_0 + dt) = c dt +  \eps(dt)dt \geq c dt -  |\eps(dt)|dt > \frac{c}{2}dt > 0
\]
That is, for $t \in (t_0, t_0+ \delta)$, it must be that $x_i(t) > x_j(t)$. \end{proof}
This is a useful tool in comparing individuals as the system evolves since it still applies when nudge effects are present. The following is an extension of Corollary 1.

\begin{cor}\label{periodBound}
Suppose that $x_i(0) = 0$, and $\sigma_A\alpha_i - \sigma_S\mu_S > 0$. Then, the time until individual $i$'s first action is less than or equal to $T = \frac{1}{A}\tanh^{-1}(\tau)$
\end{cor}
\begin{proof} Since $\gamma_i$ is always non-negative, \[
\dot x_i = [\sigma_A \alpha_i + \sigma_S(\gamma_i - \mu_S)]\sigma_C (\gamma_i + \mu_C)(1-x_i)(1+x_i) \geq A_i(1-x_i)(1+x_i)
\]
for all $t$. Define the system with only one individual given by $\dot x^* := A_i(1-x_i)(1+x_i)$ and $x^*(0) = 0$. Suppose for contradiction that individual $i$ acts after time $T$, then $x_i, x^*$ are both continuous until $T$. Applying Lemma \ref{LEMMA:derivativeDef}, we see that for $t \in (0, \delta)$, it must be that $x_i(t) > x^*(t)$. By supposition $x^*$ acts before $x_i$, so the intermediate value theorem tells us there is some (first) time $\delta < t' < T$ where $x^*(t') = x_i(t')$, and applying Lemma \ref{LEMMA:derivativeDef} again gives us the desired contradiction since $x^*$ could not possibly act before $x_i$ if $x^* \leq x_i$ \end{proof}

This tells us that nudge effects can only shorten the time until an individual's first action, which should be intuitively clear as a nudge effect can only increase the derivative. Furthermore, a re-framing of the contrapositive of this tells us something useful: if individual $i$ never acts, then $\sigma_A\alpha_i - \sigma_S\mu_S \leq 0$. This is because this term controls the sign of $\dot x_i$, so if it were positive, that individual must eventually act.

\begin{cor}\label{Cor:countablyInfinite}
{Suppose that for individual $i$, $\sigma_A\alpha_i - \sigma_S\mu_S > 0$. Then, individual $i$ will act countably infinitely many times after any fixed time $t_{start} \geq 0$}
\end{cor}
\begin{proof} $t$ is unbounded above, we know from Corollary \ref{periodBound} that individual $i$ will act at time $t^* \leq T$, then $x_i(t^* + dt)  =0$. We can then apply Corollary \ref{periodBound} at $t^* + dt$ again to see that $x_i$ will act again. We can apply this argument every time that individual $i$ acts, so it must act infinitely many times. In the rest of this proof, we show that an individual cannot act an uncountable number of times. To do so, we show that the interval between actions is bounded from below. Consider
\[
\dot x_i = [\sigma_A \alpha_i + \sigma_S(\gamma_i - \mu_S)]\sigma_C (\gamma_i + \mu_C)(1-x_i)(1+x_i)
\]
The presence of non-zero $\gamma_i$ only results in a larger derivative, and $\max \gamma_i = 1$. $(1-x_i)(1+x_i)$ is maximized at 0, where it is 1. Then, \[
\dot x_i \leq [\sigma_A \alpha_i + \sigma_S - \sigma_S\mu_S]\sigma_C (1 + \mu_C)
\]
Regardless of $x_i(0)$, the first time this individual acts at time $t$, we will have $x_i(t + dt) = 0$, then the variable interval between actions $T_{var}$ is bounded as follows \[
T_{var} \geq \frac{\tau}{[\sigma_A \alpha_i + \sigma_S - \sigma_S\mu_S]\sigma_C (1 + \mu_C)}
\]
and the RHS is constant. \end{proof}

The last result here gives us a condition for a state where all individuals must be acting.

\begin{prop}\label{smallestAlpha}
  Suppose we have $x(0) = y(0) = 0$ for all individuals. Suppose the individual $i$ with the smallest $\alpha$ acts at time $t$, then all other individuals must also act. Furthermore, all other individuals could not have acted later than individual $i$.
\end{prop}

\begin{proof} We start by considering $n=2$. If $\alpha_1 = \alpha_2$, then the dynamics for each individual are identical and the claim is trivial. Then, without loss of generality, let $\alpha_1 < \alpha_2$ and for all $t \geq 0$, $x_2(t) < \tau$. Let $x_1(0) = x_2(0) = y_1(0) = y_2(0) = 0$. Suppose for contradiction that there exist time(s) $t^* \geq 0$ where $x_1(t^*) = \tau$, and for simplicity, say that $t_0$ is the smallest such. Note first that with the initial conditions above,
\[
\dot x_i(0) = [\sigma_A \alpha_i - \sigma_S \mu_S]\sigma_C\mu_C\]
This implies that $\dot x_1(0) < \dot x_2(0)$. Applying Lemma \ref{LEMMA:derivativeDef}, we see that all $t \in (0, \delta)$ satisfy $x_2(t) > x_1(t)$.

When neither individual has acted, the system is continuous. Since $x_1(t) < \tau$ for all $0 \leq t \leq t_0$, and $x_2(t) < \tau$ for all $ t \geq 0$, both by assumption, there exists at least one time where $x_1(t) = x_2(t)$ by the intermediate value theorem. Call $t' \leq t_0$ the smallest one. Note that $y_i(t) = 0$ for $t < t_0$, so $\dot x_2(t') > \dot x_1(t')$. Again, we use Lemma \ref{LEMMA:derivativeDef} to see that $dt \in (0,\delta')$ satisfies \[
x_2(dt + t') > x_1(dt + t')
\]
thus $x_1(t) > x_2(t)$ is impossible for any $t \leq t_0$, which is a contradiction. So we have shown that $x_1(t) = \tau$ for any $t$ implies $x_2(p) = \tau$ for some $p \leq t$ given the conditions in the supposition. It is easy to show that this generalizes to the case with $n$ individuals -- just apply the same argument above to each of $x_i, \; i \geq 3$. We can do this because for every individual who hasn't yet acted, $\gamma_i$ is the same, and they all had $x_i(0) = 0$. \end{proof}
We can trivially relax the initial conditions somewhat while following the same proof. $x_2(0) \geq x_1(0)$ and $y_2(0) \geq y_1(0)$ are sufficient rather than exactly having $x_1(0) = x_2(0) = y_1(0) = y_2(0) = 0$. We could reword the same proposition with the claim that an individual with a larger $\alpha$ takes its first action at least as fast as an individual with a smaller $\alpha$.

\section{Classification of Dynamics}\label{sec:ClassificationOfDynamics}

We split the general behavior of the system into three cases that are explored qualitatively in \cite{schwarze2024planned}. These are the no action state, the partial action state, and the full action state. In the no action state, no individual ever acts. In the second, some of the individuals act, and if an individual ever acts, it acts countably infinitely many times after any fixed time $t_{start}$. In the full action state, all of the individuals act countably infinitely many times after any fixed time $t_{start}$. For all of these, we make the natural assumptions that $x_i(0) < \tau$ and $y_i(0) = 0$.

\begin{prop}\label{Prop:noAction}
Suppose $\gamma_i(t) = 0$ for all $t$, $x_i(0) < \tau$, and  $\sigma_A\alpha_i - \sigma_S\mu_S \leq 0$, then individual $i$ will never act.
\end{prop}
\begin{proof} $\dot x_i \leq 0$ for all $t$, so $x_i(t) = \tau > 0$ is impossible. \end{proof}

\begin{cor}\label{Cor:noActionState}
    A system falls under the ``no action'' classification (Fig \ref{fig:NoAction}) if and only if the individual $i$ with the largest $\alpha$ satisfies $\sigma_A\alpha_i - \sigma_S\mu_S \leq 0$ and $x_j(0) < \tau$ for all individuals.
\end{cor}
\begin{proof} If the system is classified in the no action state, then the contrapositive of Corollary \ref{periodBound} tells us that  $\sigma_A\alpha - \sigma_S\mu_S < 0$ is satisfied for each individual. For the other direction, $\alpha_i$ satisfies the condition above. Suppose for contradiction that another individual acts. Pick the individual that acted first, then Proposition \ref{smallestAlpha} tells us that this must be individual $i$ as it had the largest $\alpha$. But if $i$ was the first individual to act, then $\gamma_i =0$, which contradicts Proposition \ref{Prop:noAction}. \end{proof}

\begin{figure}
    \centering
    \includegraphics[width=0.8\linewidth]{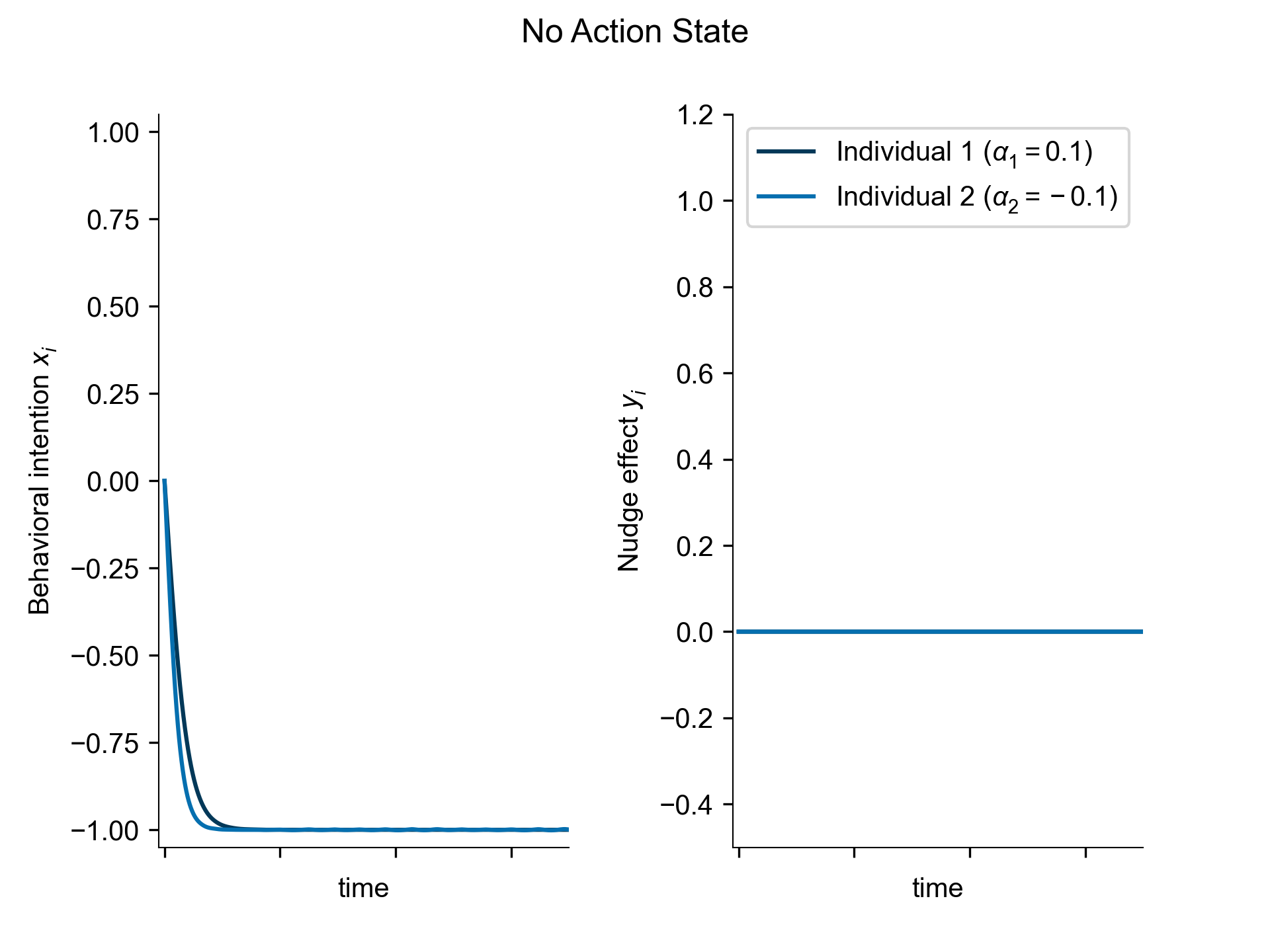}
    \caption{The No action state in which both individuals have initially decreasing behavioral intentions}
    \label{fig:NoAction}
\end{figure}

\begin{prop}\label{Prop:fullActionEasy}
    If $\sigma_A\alpha_i - \sigma_S\mu_S > 0$ for all $i$, then the system is classified under the full action state (Fig \ref{fig:FullActionEasy}
    ).
\end{prop}
\begin{proof} This follows directly from Corollary \ref{Cor:countablyInfinite}. \end{proof}

\begin{figure}
    \centering
    \includegraphics[width=0.8\linewidth]{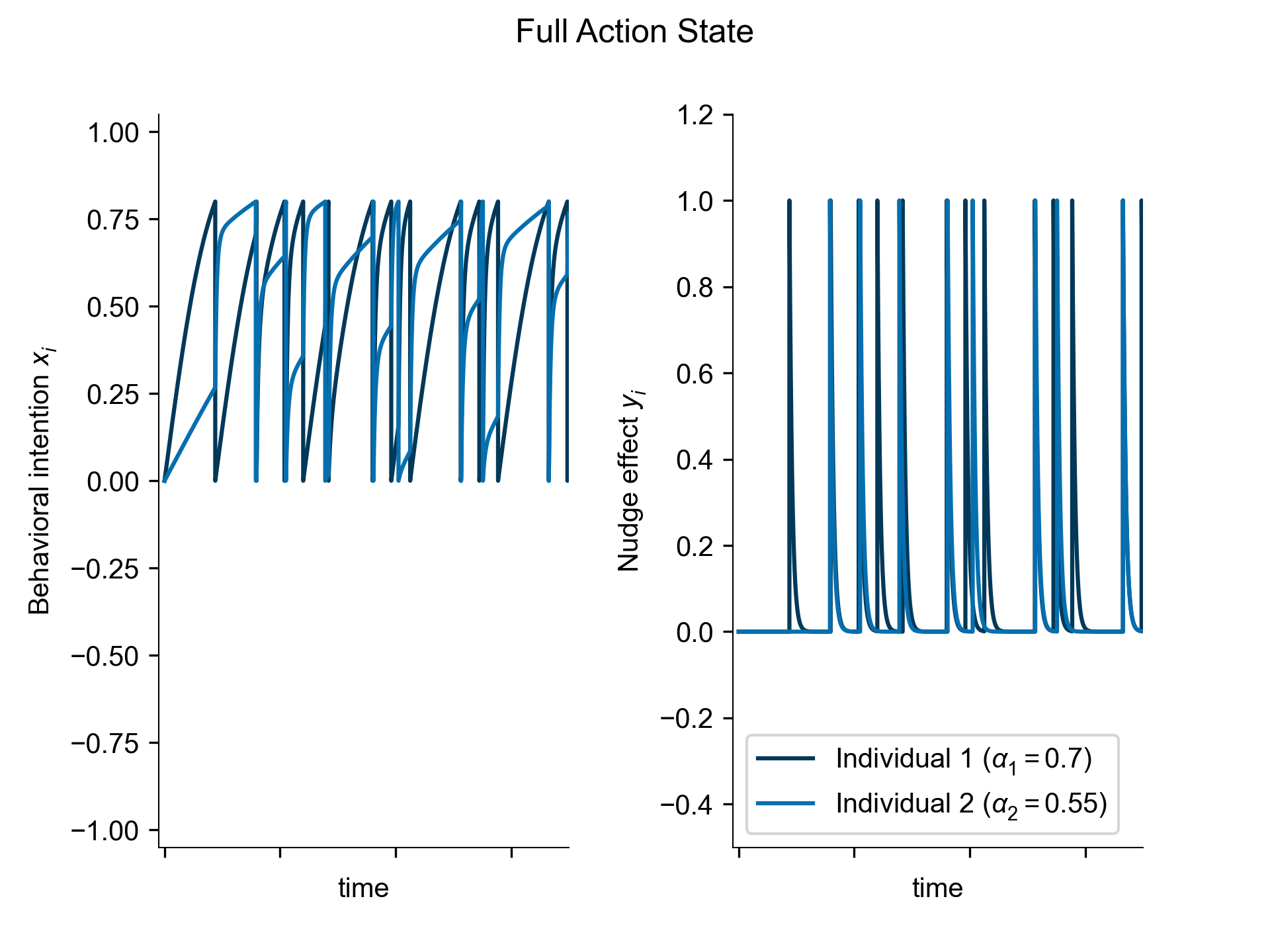}
    \caption{The full action state in which both individuals have initially increasing behavioral intentions and so both will continue to act indefinitely}
    \label{fig:FullActionEasy}
\end{figure}

It is analytically intractable to obtain sufficient and non-trivial necessary conditions for the partial action and full action statement in the case of $n$ individuals, so from here, we are going to classify the general behavior of the system in the $n=2$ case. We can cover most of the cases from the results earlier. Now, we focus on when, without loss of generality, $\sigma_A\alpha_1 - \sigma_S\mu_S > 0$, but $\sigma_A\alpha_2 - \sigma_S\mu_S \leq 0$. We now must develop tools to classify systems of this type into the partial action state and full action state.

In the following sections, we often make basic, routine assumptions (hypotheses), which are listed below: \[
\begin{cases}
    \sigma_A\alpha_1 - \sigma_S\mu_S > 0 & (h_1)\\
    \sigma_A\alpha_2 - \sigma_S\mu_S \leq 0 & (h_2) \\
    x_1(0)= y_1(0) = y_2(0) =0  & (h_3) \\
    x_2 \text{ has not yet acted} & (h_4) \\
\end{cases}
\]
It should be clear why we're generally assuming the first two -- this is the last case we need to consider for the classification. We will eventually relax the fourth assumption and show that if $x_2$ acts, then it will act again. The third assumption is explained in the following section. Also note that $(h_4)$ implies that $x_2(0) < \tau$.

\subsection{Derivation of an Explicit Solution}\label{subsec:Derivation}

Suppose $(h_1), (h_3), (h_4)$. From Corollary \ref{period}, call the first time that $x_1$ acts as $T$. Before $x_2$ has acted, which is $(h_5)$, we call $T$ the \textit{period} of $x_1$, as $x_1$ acts every at positive integer multiple of $T$. This is why we need the third assumption -- otherwise, the useful notion of a period becomes more difficult to handle.

The first tool we derive is an explicit expression for $x_2$, valid until either one period of $x_1$ elapses or $x_2$ acts.

\begin{prop}\label{Prop:explicitSolution}
    Suppose $(h_1), (h_3), (h_4)$. Then, we can write $x_2$'s position after $t$ time, given a starting position of $u_0$ at starting time $kT$, as $x_2(t, u_0) = \tanh(At+Be^{-rt} + Ce^{-2rt} + D(u_0))$. This is valid in $[kT, \min((k+1)T, t^*)]$, where $t^* > t'$, $k \geq 1$ and $x_2(t^*) = \tau$.

    $A$ is as in Proposition 2, $B=-\frac{1}{r}[(\sigma_A\alpha_2-\sigma_S\mu_S)\sigma_C+\sigma_S\sigma_C\mu_C]$, $C=-\frac{\sigma_S\sigma_C}{2r}$, and $D(u_0)=\tanh^{-1}(u_0)-B-C$.
\end{prop}
\begin{proof} In the two individual system, immediately after $x_1$ acts, the system is
\begin{equation}
    \begin{split}
        \dot{x}_1&=[\sigma_A\alpha_1+\sigma_S(y_2-\mu_S)][\sigma_C(y_2+\mu_C)](1-x_1^2)\\
        \dot{x}_2&=[\sigma_A\alpha_2+\sigma_S(y_1-\mu_S)][\sigma_C(y_1+\mu_C)](1-x_2^2)\\
        \dot y_1 &= -ry_1\\
        \dot y_2 &= -ry_2
    \end{split}
\end{equation}
And the initial conditions are $x_1(0)=0,\; x_2(0)= u_0,\; y_1(0)=1,\;y_2(0)= \wt y $, where $u_0$ is the position of $x_2$ when $x_1$ acts, and $\wt y$ similarly. Notice that for our constraints, the system is decoupled, so we need only consider $x_2$ and $y_1$. Furthermore $y_1(t)=e^{-rt}$ so we have a single separable ODE. The computation follows through grouping and keeping track of the constants:
\[\dot{x}_2=[C_1+C_2e^{-rt}+C_3][C_4e^{-rt}+C_5](1-x_2^2)\]
where $C_1=\sigma_A\alpha_2, \;C_2 = \sigma_S,\; C_3=-\sigma_S\mu_S,\; C_4=\sigma_C$, and $C_5=\sigma_C\mu_C$. Simplifying:
\[\dot x_2 = [C_6+C_7e^{-rt}+C_8e^{-2rt}](1-x_2^2)\]
where $C_6= A_2 = (\sigma_A\alpha_2-\sigma_S\mu_S)(\sigma_C\mu_C),\; C_7 = (\sigma_A\alpha_2-\sigma_S\mu_S)\sigma_C+ \sigma_S\sigma_C\mu_C,\; C_8=\sigma_S\sigma_C$. For brevity we write $x_2$ just as $x$.
\begin{equation*}
    \begin{split}
        \frac{1}{1-x^2}\frac{dx}{dt}&=C_6+C_7e^{-rt}+C_8e^{-2rt}\\
        \int\frac{1}{1-x^2}\frac{dx}{dt}dt&=\int(C_6+C_7e^{-rt}+C_8e^{-2rt})dt\\
        \int\frac{1}{1-x^2}{dx}&=\big[C_6t-\frac{C_7}{r}e^{-rt}-\frac{C_8}{2r}e^{-2rt}\big]+C_0
    \end{split}
\end{equation*}
where $C_9=\frac{1}{r}[(\sigma_A\alpha_2-\sigma_S\mu_S)\sigma_C+\sigma_S\sigma_C\mu_C]$, $C_{10}=\frac{\sigma_S\sigma_C}{2r}$, and $C_0$ is the constant of integration. Through partial fraction decomposition on the LHS, we obtain the following
\[\frac{1}{2}\ln\lt(\frac{1+x}{1-x}\rt)=\big[C_6t-\frac{C_7}{r}e^{-rt}-\frac{C_8}{2r}e^{-2rt}\big]+C_0\]
As $x$ is bounded in $(-1,1)$, we need not use the absolute values. Using this, we can solve for $C_0(u_0)$ through $x(0)=u_0$, which gives us
\[C_0(u_0)=\frac{1}{2}\ln\lt(\frac{1+u_0}{1-u_0}\rt)+\frac{C_7}{r}+\frac{C_8}{2r}\]
We rearrange the equation again to get that
\[\frac{1+x}{1-x}=\exp\Big(2\big(\underbrace{C_{6}t-\frac{C_{7}}{r}e^{-rt}-\frac{C_{8}}{2r}e^{-2rt}+C_0(u_0)}_{h(t,u_0)}\big)\Big)\]
The left-hand side of the function is injective from $(-1,1)\to(0,\infty)$ so we can solve for the unique solution
\[x=\frac{e^{2h(t,u_0)}-1}{e^{2h(t,u_0)}+1}=\tanh\lt(h(t,u_0)\rt)\]
So the solution is written as follows
\[x(t;u_0) =\tanh(At+Be^{-rt}+Ce^{-2rt}+D(u_0))\]
where $A= A_2 = (\sigma_A\alpha_2-\sigma_S\mu_S)(\sigma_C\mu_C),\; B=-\frac{1}{r}[(\sigma_A\alpha_2-\sigma_S\mu_S)\sigma_C+\sigma_S\sigma_C\mu_C],\; C=-\frac{\sigma_S\sigma_C}{2r},\; D(u_0)=\frac{1}{2}\ln\lt(\frac{1+u_0}{1-u_0}\rt)-B-C=\tanh^{-1}(u_0)-B-C$ \end{proof}
As a reminder, this solution describes the position of $x_2$ at time $kT + t$, where $k\geq 1$, given starting position $u_0$ at time $kT$. This is valid until whichever happens first: $(k+1)T$ or a time where $x_2$ crosses the threshold.

\subsection{A Useful Invariant}\label{subsec:Invariant}

\begin{prop}\label{Prop: invariant}
    Fix $(*)$ as one of $<, \, =, \text{ or } >$.  Take $(h_1)$-$(h_4)$ and suppose $x_2$ is not allowed to act, that is, observe the dynamics of $x_2$ as if this particular individual did not experience the threshold. Further suppose that \[x_2(0; u_0) = u_0 \; (*) \;  x_2(T;u_0)\] Then, taking any $v_0 \in (-1,1)$, $x_2$ will satisfy \[
    x_2(0; v_0) = v_0 \; (*) \;  x_2(T;v_0)
    \]
\end{prop}
Before giving a proof, it is helpful to interpret the result. The statement says that while $x_2$ has not yet acted, $x_2$'s behavior on any given period of $x_1$ -- net increase, decrease, or no change over the course of one nudge effect (ignoring the threshold) -- is entirely independent of $x_2$'s starting position. 

So $x_2$ has not yet acted, if it has experienced a net increase on any given period of $x_1$, it will experience a net increase over \textit{all} periods of $x_1$. We can say the same for a net decrease or no net change over any given period. This idea will be key in the future.

\begin{proof} The supposition written out is that $u_0 <  \tanh(AT+Be^{-rT}+Ce^{-2rT} + D(u_0))$. Since $\tanh^{-1}$ is monotonic, we have that \[
\tanh^{-1}(u_0) <  AT +Be^{-rT}+Ce^{-2rT} + D(u_0)
\]
Recall $D(u_0) = \tanh^{-1}(u_0) - B - C$, so we have that \[
 B+ C<  AT+Be^{-rT}+Ce^{-2rT} 
\]
which is an expression entirely independent of $u_0$. Every step was reversible, so we can work backwards from  \[
 B+ C<  AT+Be^{-rT}+Ce^{-2rT} \Longleftrightarrow \tanh^{-1}(v_0) < AT + Be^{-rT}+ Ce^{-2rT} + \tanh^{-1}(v_0) - B - C
\]
\[
=  AT + Be^{-rT}+ Ce^{-2rT} + D(v_0)
\] 
and the above two lines imply that \[
x_2(0;v_0) < \tanh\lt(AT + Be^{-rT}+ Ce^{-2rT} + D(v_0)\rt) = x_2(T;v_0)
\]
Replacing $<$ by $=$ and $>$ respectively, we get \[
B+ C =   AT+Be^{-rT}+Ce^{-2rT} \quad \text{ and } \quad B+ C > AT+Be^{-rT}+Ce^{-2rT}
\]
\end{proof}
Regardless if $x_2$ acts during any given period, we can still write and compare $B + C$ with $AT+Be^{-rT}+Ce^{-2rT}$ as an invariant. That is, this comparison still tells us about the behavior of $x_2$ on every period. We will use this to rule out the possibility of some undesirable behavior by examining $x_2$ as if it ignored the threshold and deriving a contradiction. 

In the following sections, we will show that $x_2$'s action is possible only when $B+C < AT+Be^{-rT}+Ce^{-2rT}$, and, in fact, the full action state is guaranteed in this case. Since we use similar conditions so often, we're going to write $M := AT+Be^{-rT}+Ce^{-2rT} - B - C$.

Thus, for the hypotheses $(h_1)$-$(h_4)$, our ultimate classification is that $M \leq 0$ is exactly the partial action state, and $M > 0$ is exactly the full action state.

\subsection{Classification of Partial Action}\label{subsec:PartialAction}
Assuming $y_i(0) = 0$ and $x_1(0) = 0, x_2(0) < \tau$, which have been roughly the standard conditions thus far, we can make some general observations about the shape of the curve throughout the period $[kT, (k+1)T]$ for $k \geq 1$. When $k = 0$, $x_1$ has not acted yet and $x_2$ must be decreasing in all cases. 

Suppose $M < 0 $, it is possible that $x_2$ still fails to increase on any interval within the period. Any possible increase must occur at the beginning of the interval as the nudge effect $e^{-rt}$ contribution in \[
[\sigma_A\alpha_2 + \sigma_S(e^{-rt} - \mu_S)] \quad \text{ (recall this controls the sign of } \dot x_2)
\]
is decreasing. By supposition, the end of the period marks a net decrease from the start of the period. This implies that $x_2$ must be decreasing by the end of the period. Then, in the case that $\dot x_2(T) > 0$, there must be a local maximum, which we find in \ref{app:criticalpoint}.

When $M = 0$, $x_2$ must increase before it decreases because $x_2(kT) =  x_2((k+1)T)$, and the nudge effect $e^{-rt}$ contribution in \[
[\sigma_A\alpha_2 + \sigma_S(e^{-rt} - \mu_S)]
\]
is decreasing. However, $x_2$ remains unchanged after each period.

Lastly, when $M>0$, it is possible that $x_2$ never decreases on the period, in which case we may not be able to find a local maximum. If it does decrease, then we can. Note that since this marks a net increase over the period, every point on the interval is larger than the first point. 

We stress that in all cases, there is at most one critical point. There can never be more than one critical point as the derivative of $x_2$ is always decreasing over the period. 

In the following corollaries, we show that given $(h_1)$-$(h_4)$, the partial action state occurs if and only if $M \leq 0$. To do so, we show that that in either the equality or inequality case, if there is a local maximum, then it must be less than $x_2(0)$ and is therefore less than the threshold. 

\begin{cor}\label{Cor: partialAction1}
   Take $(h_1)$-$(h_4)$ as hypotheses. Suppose $M = 0$. Then, the system is in the partial action state. That is, $x_2$ never acts. 
\end{cor}
As mentioned earlier, we'd like to show that $x_2$ has a local maximum below the threshold. We suppose for contradiction that this local maximum is above the threshold. Then, we examine $x_2$ as if this individual ignores the threshold, with the intent of applying Proposition \ref{Prop: invariant} to reach our contradiction. Below is a sketch of how we achieve the contradiction.

$x_2$ was only decreasing on $[0,T]$, and on every future period, $x_2$ is increasing during the  start of $[kT,(k+1)T]$. $x_2$ must return to $x_2(T)$ at the beginning of each period. Putting these facts together, we realize that if $x_2$ were able to cross $x_2(0)$, $|\dot x_2|$ is too small to return to $x_2(T)$ by the beginning of the next period.

\begin{proof} It should be obvious that if $x_2$ doesn't act on $[T,2T]$, then it will never act as the same behavior will be exactly replicated on $[kT, (k+1)T]$ for all $k \geq 1$. That is, there will be no net change over each interval. Suppose for contradiction that there is some point $t'$ on $[T,2T]$ where $x_2(t') \geq \tau$, then $x_2(t') > x_2(0)$ by supposition. For the moment, say that $x_2$ ignores the threshold and examine what $x_2$'s behavior would be -- it must be that $x_2(2T) = x_2(T)$ by Proposition \ref{Prop: invariant}. We show this is not possible.

To see why, note that for $t \in (0,T)$, $\dot x_2(t) < 0$ was always true. We can make the following comparison as the action of $x_1$ at $T$ can only increase the derivative of $x_2$: \[\dot x_2(T + t) > \dot x_2(t) \quad t \in (0,T)\]  Since $x_2(T) = x_2(2T)$, but for any small $\eps > 0$, $x_2'(T+ \eps) > 0$ for $x_2(t') \geq\tau$, then $x_2$ must, at some point, be decreasing on the period $(T,2T)$, so by the intermediate value theorem, there is a point where the derivative is 0. We have found an exact expression for this in  \ref{app:criticalpoint}. Call that point $T + t_{crit}$, and $x_2(T + t_{crit}) \geq x_2(t') \geq \tau > x_2(t_{crit})$. Then, \[
x_2(T + t_{crit})- x_2(2T)  > x_2(t_{crit}) - x_2(T) 
\]
since $x_2(T) = x_2(2T)$. However, $\dot x_2(T + t) > \dot x_2(t)$ for $t \in (0,T)$, so it cannot be true that $x_2(2T) = x_2(T)$. 

Restated: the derivative is \textit{larger} in sign on $(T,2T)$ than on $(0,T)$, and the position of $x_2$ at $T + t_{crit}$ is \textit{larger} than the position of $x_2$ at $t_{crit}$, but the remaining time $2T - t_{crit} - T = T - t_{crit}$ to reach the position $x_2(T)= x_2(2T)$ is equal. 
\end{proof}

This tells us that we can pick any starting value lower than the threshold and land in a unique oscillatory equilibrium. That is, $x_2$ will oscillate between $x_2(T)$ and $x_2(T + t_{crit})$ with exactly the same behavior over each period (Fig \ref{fig:InitialPosition}).

\begin{figure}
    \centering
    \includegraphics[width=\linewidth]{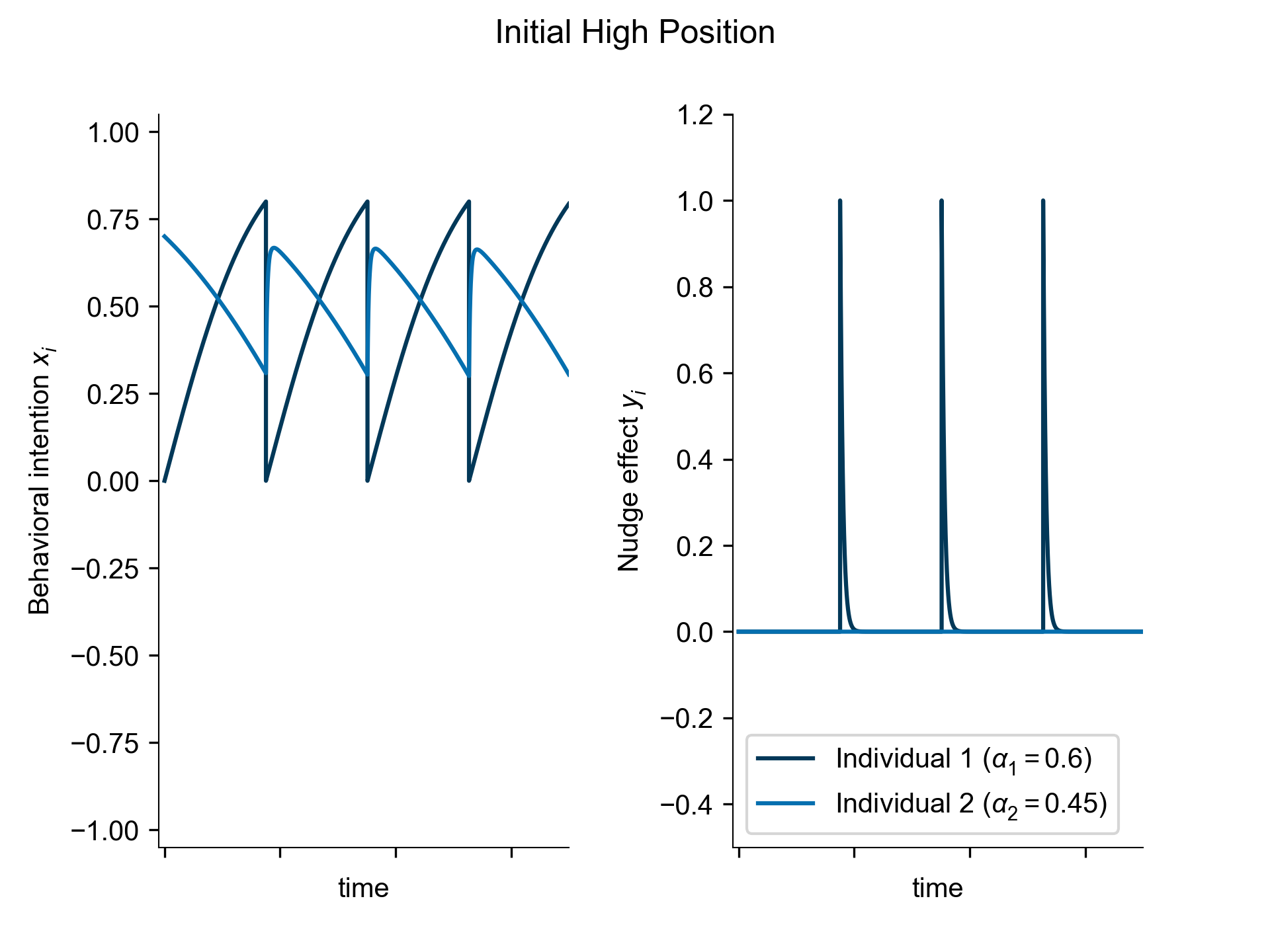}
    \includegraphics[width=\linewidth]{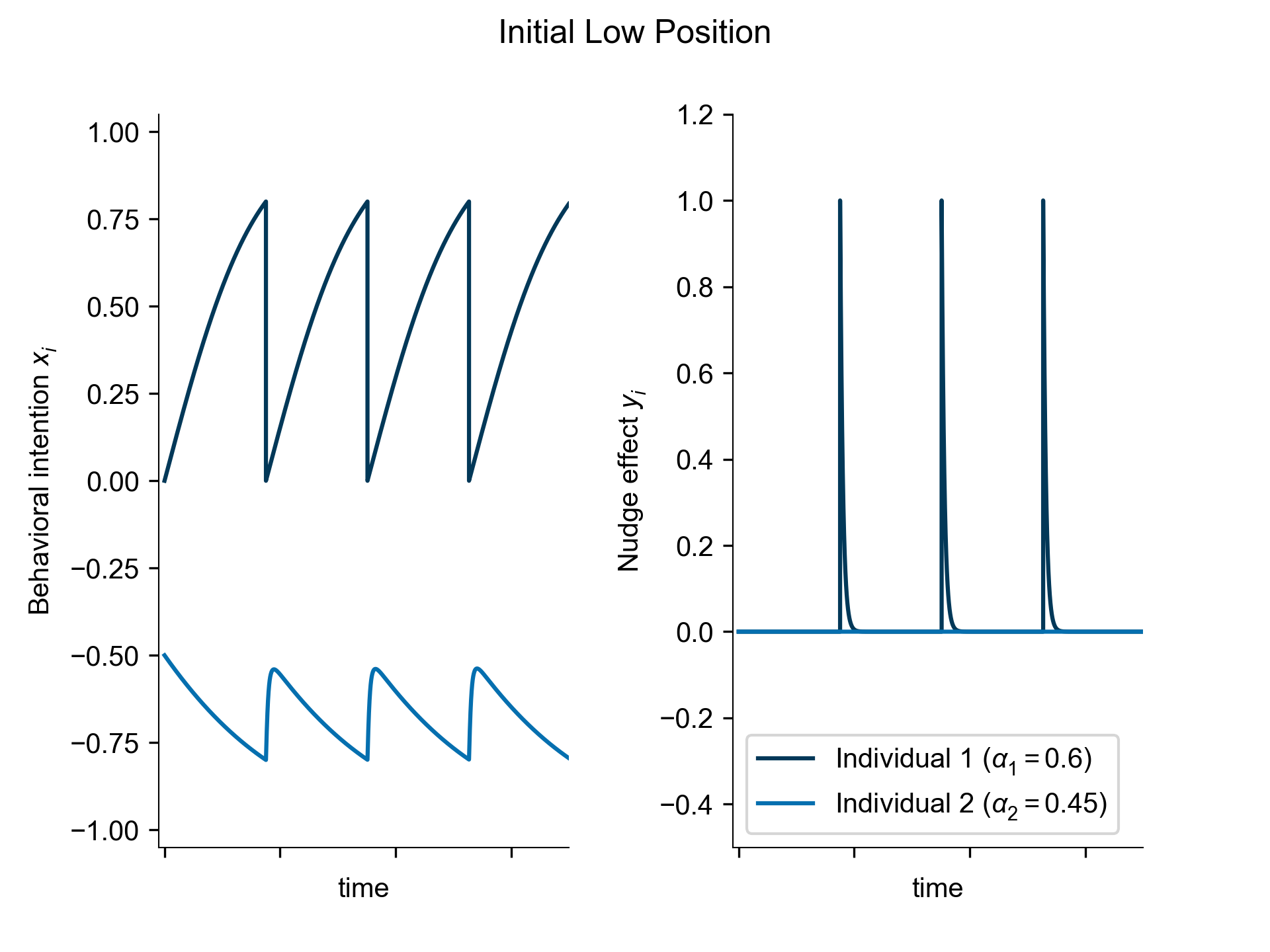}
    \caption{\textbf{Top} an example of the initial data for individual 2's behavioral intention being very high and not being able to reach the behavior threshold. This behavior is qualitatively no different than the case on the \textbf{Bottom} where individual 2's initial behavioral intention is quite low.}
    \label{fig:InitialPosition}
\end{figure}

We prove a quick useful fact before finishing the classification of the partial action case.

\begin{prop}\label{startsHigher}
    Suppose that $\alpha_i = \alpha_j$, and neither $x_i,x_j$ acts on some interval $[t_0,t_1]$. Suppose also that $x_i(t_0) \geq x_j(t_0)$, then $x_i \geq x_j$ on $[t_0,t_1]$
\end{prop}
\begin{proof} In the equality case this is obvious, so suppose that the inequality is strict. Suppose for contradiction that there is some point where $x_i(t) <  x_j(t)$, then by the intermediate value theorem, there is some point $t'$ such that $x_i(t') = x_j(t')$. Since $\alpha_i = \alpha_j$, $\dot x_i(t) = \dot x_j(t)$ on $[t',t_1]$, so $x_j(t) = x_i(t)$. \end{proof} 

\begin{cor}\label{Cor:PartialActionFull}
Take $(h_1)$-$(h_3)$ as hypotheses, and suppose $x_2(0) < \tau$. Suppose also that we have $\sigma_A\alpha_1 - \sigma_S \mu_S > 0$ and $\sigma_A\alpha_2 - \sigma_S \mu_S < 0$. The system is classified under partial action (Fig \ref{fig:netDecrease}) if and only if  $M \leq 0$ 
\end{cor}
\begin{proof} We covered the $=$ case in the last Corollary \ref{Cor: partialAction1}. This proof will be very similar. Suppose we have $M < 0$.

As in the last proof, if $x_2$ doesn't act on $[T,2T]$, it will never act. This is because $x_2$ was only decreasing on $[0,T]$, and there is a net decrease over each following period, so we can apply the preceding proposition. If $\dot x_2(T + \eps) \leq 0$ for every small $\eps > 0$, then we're done as $x_2$ will never be increasing, so suppose for contradiction that $x_2$ acts on $[T,2T]$. As a consequence, $x_2$ must be increasing on $(T, T')$ for some $T <T' < 2T$. 

We show that this contradicts $M < 0$. To see this, examine the behavior of $x_2$ as if it ignores the threshold. 

$x_2(t)$ must be decreasing over some interval $[T', 2T]$, or else we wouldn't have the net decrease prescribed by $M< 0$. Then, the intermediate value theorem tells us that there is a point where $\dot x_2  =0$, and we've derived in \ref{app:criticalpoint} that this point is $T + t_{crit}$. Call that point $T + t_{crit}$, and $x_2(T + t_{crit}) \geq x_2(t') \geq \tau > x_2(t_{crit})$. Then, \[
x_2(T + t_{crit})- x_2(2T)  > x_2(t_{crit}) - x_2(T) 
\]
since $x_2(T) > x_2(2T)$ by supposition. However, $\dot x_2(T + t) > \dot x_2(t)$ for $t \in (0,T)$, so it cannot be true that $x_2(2T) < x_2(T)$, as reasoned similarly in the last corollary. 
\end{proof}

\begin{figure}
    \centering
    \includegraphics[width=0.8\linewidth]{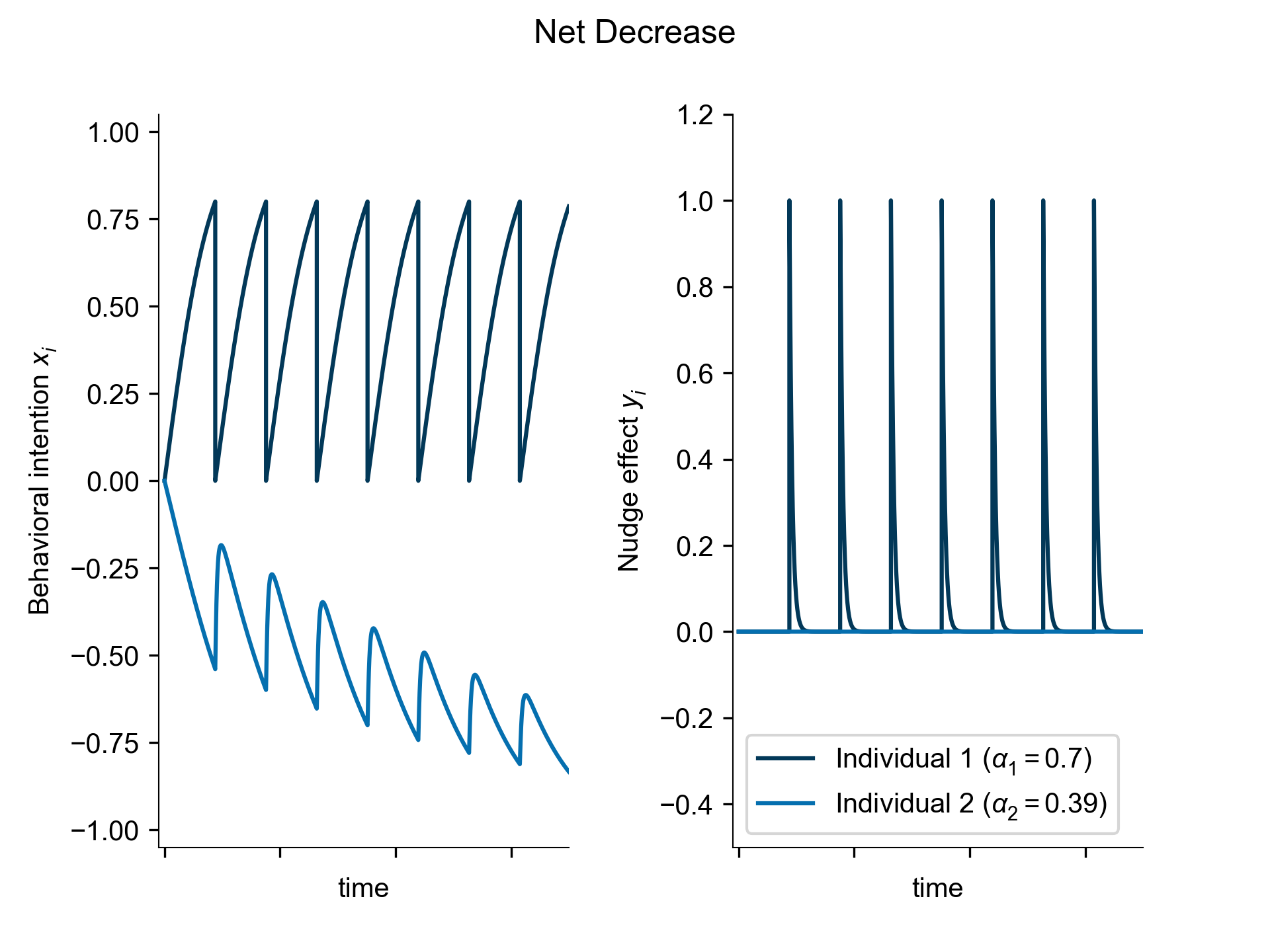}
    \caption{With each action of individual 1, individual 2's behavioral intention initially increases but decreases further than int increased by the next time individual 1 acts. This results in a net decrease for individual 2 and the solution is classified as partial action.}
    \label{fig:netDecrease}
\end{figure}

\subsection{Classification of Full Action}\label{subsec:Fullaction}

Having completely classified the no action and partial action states for the hypotheses $(h_1)$-$(h_4)$, we would like to now show that the remaining case, where we still take $(h_1)$-$(h_4)$, but $M > 0$, gives the full action state. There are two pieces to this. First, we show that $x_2$ must act once, then we show it must continue to act infinitely many times.  From Proposition \ref{Prop:explicitSolution}, we view \[
x(T, u_0) = \tanh(AT + Be^{-rT} + Ce^{-2rT} + D(u_0))\] as a function of $u_0$, where $A, B, C, D$ are the same as in that proposition.

Before beginning the next proof, it is helpful to note a shift in perspective. Rather than viewing our system as evolving in time, we are now viewing it in position. In particular, we rewrite the function above as
\[
f(x) := \tanh(M + \tanh^{-1}(x))
\]
Given some position $x$ to $f$, $f$ returns that position evaluated over one period. For example, we have $f(x_2(kT)) = x_2((k+1)T)$, so long as we ignore the threshold.

\begin{theorem} 
Suppose $(h_1)$-$(h_4)$, $M> 0$, and let $x_* = -\frac{1 - \sqrt{1 - (\tanh M )^2}}{\tanh M}$. Then, $x_* \in (-1,0)$ and \[
\begin{cases}
    f'(x) > 1  & x\in (-1, x_*) \\
    f'(x) < 1 & x \in (x_*,1) \\
    f'(x) = 1 & x = x_*
\end{cases}
\]
\end{theorem}
In words, $f$ is a contraction above $x_*$ and is expanding below $x_*$ 

\begin{proof} We have the identity that \[\tanh(x + y)  = \frac{\tanh x + \tanh y }{1 + \tanh x\tanh y}\]
and applying this to $f$ gives us that
\[
f(x) = \frac{\tanh(M) + x}{ 1 + x\tanh(M)}
\]
We consider the function $g(x) = f(x) -x$ because $g'(x) = f'(x) - 1$, thus we can analyze the sign of $g'(x)$. Writing out $g$ explicitly gives
\[
g(x) = f(x) - x = \frac{x + \tanh(M)}{ 1+ x\tanh(M)} - x = \frac{\tanh(M)(1 - x^2)}{1 + x\tanh(M)}
\]
and its derivative as
\[
g'(x) = -\frac{\tanh(M)(2x + \tanh(M)(1+x^2))}{(1 + x\tanh(M))^2}
\]
Since $\tanh(M) > 0 $ and $(1+x\tanh(M))^2 > 0$, the sign is \[
-\text{sgn}(2x + \tanh(M)(1+x^2))
\]
Looking for where this is 0, we need only apply the quadratic formula,
and one solution is outside the valid values of $x_i$: \[
-\frac{1  + \sqrt{1 - (\tanh M)^2}}{\tanh M} < -1
\]
so we're left with 
\[
x_* = -\frac{1  - \sqrt{1 - (\tanh M)^2}}{\tanh M}
\]
We show that $x_* \in (-1,0)$ in the Appendix. It is a straightforward computation -- for now, take it as a fact. Recall that since $M> 0$, we have $\tanh(M) > 0$. Note that \[\text{sgn}(g'(x)) = -\text{sgn}((2x + \tanh(M)(1+x^2)) = -\text{sgn}(\tanh(M) + 2x + \tanh(M)x^2)\] The argument of the last expression is an upward-facing parabola. We know this is negative between the roots, that is, for $x \in (-1,x_*)$ and positive for $x \in (x_*,1)$. Then, \[
 \begin{cases}-\text{sgn}(2x + \tanh(M)(1+x^2))  > 0 & x \in (-1, x_*)\\
 -\text{sgn}(2x + \tanh(M)(1+x^2))  < 0  & x \in (x_*, 1) \\
-\text{sgn}(2x + \tanh(M)(1+x^2))  = 0 & x = x_*
\end{cases} \]implies that for each case above we have correspondingly \[ \begin{cases}
  g'(x) > 0 & x \in (-1, x_*)   \\
g'(x)  < 0  & x \in (x_*, 1)\\
g'(x) = 0 & x = x_*
\end{cases}\]
Rewriting $g'(x) = f'(x) - 1$, we obtain the desired result\[
\begin{cases}
  f'(x) > 1 & x \in (-1, x_*)   \\
f'(x)  < 1  & x \in (x_*, 1)\\
f'(x) = 1 & x = x_*
\end{cases}\]
\end{proof}
The next two results establish that $x_2$ must act. The first gives a loose bound on how long it takes $x_2$ to cross $x_*$, and the second gives another loose bound on how long it takes $x_2$ to cross the threshold from $x_*$. 

Note that $x_1$ first acts at $T$, and by $(h_2)$, $x_2$ is certainly decreasing until $T$. We want to show that there is a period during which either $x_2 > x_*$ throughout the entire period, or $x_2$ acts on that same period where it crosses $x_*$. In that first case, it suffices to show that $x_2$ is larger only at the start of some period. 
\begin{lemma}
   Given $(h_1)$-$(h_4)$, $M>0$, there exists time $mT$ for $m \geq 1$ s.t. either $x_2(mT) > x_*$ or $x_2(t_1) = x_*$ for $t_1 \in [mT, (m+1)T]$ and $x_2(t_2) = \tau$ for $t_1 \in (mT, (m+1)T]$
\end{lemma}

\begin{proof} Note that by $(h_2)$ and $(h_4)$, $\tau > x_2(0) > x_2(T)$, which is why $m \geq 1$. If $x_2(T) > x_*$, then we're done. We know that $f(x_*) > x_*$ since $f$ is monotonic increasing, so if we ever have that $f^n(x_2(T)) = x_*$, then either $x_2((n+1)T) > x_*$ or $x_2$ acted in $(nT, (n+1)T)$

So take the remaining case where $x_2(T) < x_*$. Suppose for the rest of this proof that $x_2$ does not act and that $f(x_2(T)) < x_*$ -- otherwise, we'd be done. We will find such an $m$ as described in the statement.

Write $g(x) := f(x) - x$ as earlier, and note that for $x \in (-1,x_*)$, we have that $g'(x) > 0$, so it is monotonic increasing. Then, for all $k \geq 2$ satisfying $f^k(x_2(T)) < x_*$, since $f$ is also monotonic increasing, we have
     \[ g(f^{k-1}(x_2(T))) > g(x_2(T)) \implies f^k(x_2(T)) - f^{k-1}(x_2(T)) > f(x_2(T)) - x_2(T) \]
Then, we can upper bound the possible integer $k$ satisfying $f^k(x_2(T)) < x_*$ as follows: \[
\frac{|x_* - x_2(T)| }{f(x_2(T)) - x_2(T)} \geq k  
\]
So we can pick \[
m = \lt \lceil \frac{|x_* - x_2(T)| }{f(x_2(T)) - x_2(T)}\rt \rceil
\]
as desired.

\end{proof}
This bound, $mT$, on the time is not tight. The proof used that the distance between successive applications of $f$ was no smaller than the distance $f(x(T)) - x_2(T)$, but there are many cases where this is smaller, that is, \[
f^k(x_2(T)) - f^{k-1}(x_2(T)) \gg f(x_2(T)) - x_2(T)
\]
Qualitatively, the bound is worse when $k$ is larger. If $k$ is especially small, then the difference in growth might only mean that we overestimate by a few applications of $f$. 

\begin{theorem}\label{thm:actsAgain}
    Given $(h_1)$-$(h_4)$, $M > 0$, take $m$ as in Lemma 2. Then, there must be some $n \geq 0$ where $x_2((m+n)T) = \tau$. 
\end{theorem}
\begin{proof}
   Using the last result, we can guarantee that either $x_2$ acts, or there is time $mT$ where $x_2(mT) > x_*$. Suppose for contradiction that $x_2$ does not act. 
   
   Then, there must be some value $\tau^* \leq \tau$ such that for all $k \geq 0$, we have that $f^k(x_2(mT)) < \tau^*$. But we have that $g(x):= f(x) - x$ is monotonic decreasing on $(x_*,1)$ since $g'(x) < 0$ on this interval by Theorem 1. Then for all $k \geq 2$, we have that $g(\tau) < g(f^k(x_2(mT))$, which implies \[
   f(\tau) - \tau  < f^{k+1}(x_2(mT)) - f^k(x_2(mT))
   \]
   This means that that each iteration of $f$ increases the position by no less than it does at $\tau^*$, thus we can achieve a contradiction noting that if we take any $k$ such that \[
   k \geq \frac{\tau - x_2(mT)}{f(\tau) - \tau}
   \]
   Then, choose $n$ to be \[
   n = \lt \lceil \frac{\tau - x_2(mT)}{f(\tau) - \tau} \rt \rceil
   \]
   If one wishes for a more concrete (but worse) bound, it is easy to see we can instead take \[
   n = \lt \lceil \frac{\tau - x_*}{f(\tau) - \tau} \rt \rceil
   \] 
\end{proof}
Note that the lower bound we've obtained on $t$ satisfying $x_2(t) \geq \tau$ is not tight -- it requires $x_2((m+n)T) > \tau$, but $x_2$ could have very well achieved $x_2(\ell T + t_{crit}) > \tau$ for $\ell \ll m+n$. That is, $x_2 \geq \tau$ could have occurred during an interval, rather than having to \textit{end} a particular interval above $\tau$. In fact, if $x_2$ was non-increasing at the end of each interval, then it is guaranteed that it would have crossed the threshold during $(\ell T, \ell T + t_{crit}]$. Furthermore, the bound suffers from the same issues as discussed after the last result. 

There is an alternate proof of this result, which we have put in the Appendix. That method first continuously extends $f$ to take values in $[-1,1]$. Then, we make use of the form of the Banach fixed point theorem that requires only compactness and $f'(x) < 1$ rather than a fixed coefficient $c <1$ satisfying $|f(x) - f(y)| < c|x-y|$. This proof emphasizes more of the nature of the map $f$, and we find it provides useful intuition about $f$. However, there is no bound at all on how long it takes $x_2$ to act. 

So far, we've only shown that $x_2$ will act once given the conditions $(h_1)-(h_4)$. To show the full action state, we must relax $(h_4)$. That is, we show that $x_2$ will act, given that it has already acted before. The main idea behind this final theorem in the classification is that $x_2$'s action(s) cause(s) a nudge effect on $x_1$, but $x_1$ can never act slower. Thus, the ``varying period'' in which $x_2$ receives nudges from $x_1$ is never longer than it was before $x_2$ had acted originally. That is to say, $x_2$ receives nudges from $x_1$ \textit{at least as} frequently as it did before $x_2$ acted. Then, regardless of wherever $x_2$ is evaluated on this varying period, it will still be experiencing some lower-bounded net increase. 

\begin{theorem}\label{thm:FullAction}
    Given $(h_1)$-$(h_3)$, $M>0$, $x_2$ will act infinitely many times, thus establishing the full action state (Fig\ref{fig:netIncrease}). 
\end{theorem}
\begin{proof} Suppose that $x_2$ has just acted at time $t^*$. We show it must act again. Since $M > 0$, we know that if $x_2$ receives a nudge from $x_1$ at some time $t'$, then $x_2(t') < x_2(t)$ for any $t$ in $(t', t' +T )$. Since $x_2$ acted, the time between $x_1$'s actions are all less than or equal to $T$ by Corollary \ref{periodBound}. Call this varying time $T_{var}$. From Corollary \ref{Cor:countablyInfinite}, we have a lower bound on $T_{var}$, which we'll name $T'$. We now can split into cases. 

\textbf{Case 1:} Suppose that before its first action, $\dot x_2(t) \geq 0$ for all $t$ in each period. Then, $T \leq t_{crit}$. We can conclude $T_{var} \leq T \leq t_{crit}$, so we still have $\dot x_2(t) \geq 0$ between actions of $x_1$. Because of this, we have that \[x_2(t' + T') - x_2(t') \leq x_2(t' +t) - x_2(t')\] for all $t \in [t' + T', t' + t_{crit}]$. In words, $\dot x_2 >0$ is always true between $x_1$'s actions, so the smallest net increase occurs when the least time has elapsed.

Then, let $z = \max(|x_2(T)|, \tau)$, and every net increase experienced by $x_2$ will be at least as large as $x_2(z + T') - z$ as argued similarly in Lemma 2. So, we can upper bound the number of actions of $x_1$ required for $x_2$ to act again by \[
    \lt \lceil \frac{\tau - x_2(T)}{x_2(z + T') - z} \rt \rceil
    \]
In summary, $x_2$ must act again. 

\textbf{Case 2:} Suppose that before its first action, $x_2$ crossed the critical time $t_{crit}$ in each period. Then, $T > t_{crit}$. As such, it is possible that we have $T_{var}$ both larger and smaller than $t_{crit}$ depending on the particular interval. However, we can use a similar strategy as in the previous case. Taking $T'$ and $t'$ to be as before, it is possible that \[x_2(T + t') - x_2(t') < x_2(T' + t') - x_2(t')\]
In words, it is possible that, since $x_2$ is decreasing at some point during each interval, it would experience a smaller net increase after $T$ time rather than after $T'$ time. One of these two times must result in the smallest possible net increase by any $T_{var}$. To see why, consider the following. Before the critical time, $x_2$ is always increasing, so the smallest increase occurs when it has had the least time \textit{to} increase, which we obtain using $T'$. After the critical time, $x_2$ is always decreasing, so the smallest increase occurs when it has had the \textit{most} time to decrease, which obtain using $T$. 

Using the same technique as in Proposition \ref{Prop: invariant}, we have that if $x_2(T + t') - x_2(t') < x_2(T' + t') - x_2(t')$ for some initial position $x_2(t')$, then it is true for every initial position. As a sketch, the above expression clearly shows that $x_2(T + t') < x_2(T' + t')$, and by taking $\tanh^{-1}$ of both sides, we obtain an inequality that is independent of initial position. We could follow the same process if $x_2(T + t') - x_2(t') \geq x_2(T' + t') - x_2(t')$. In either case, we've shown that if $T$ results in a smaller increase than $T'$ for some initial position, then $T$ always results in a smaller increase regardless of initial position (and vice-versa). 
    
Now, let $z = \max(|x_2(T)|, \tau)$ as before, and by the same reasoning as in the previous case, every net increase experienced by $x_2$ will be at least as large as $\min(x_2(z + T') - z, \; x_2(z + T) -z)$. So, we can upper bound the number of actions of $x_1$ required for $x_2$ to act again by \[
    \lt \lceil \frac{\tau - x_2(T)}{\min(x_2(z + T') - z, \; x_2(z + T) -z)} \rt \rceil\] In summary, $x_2$ must act again. 

\end{proof}

\begin{figure}
    \centering
    \includegraphics[width=\linewidth]{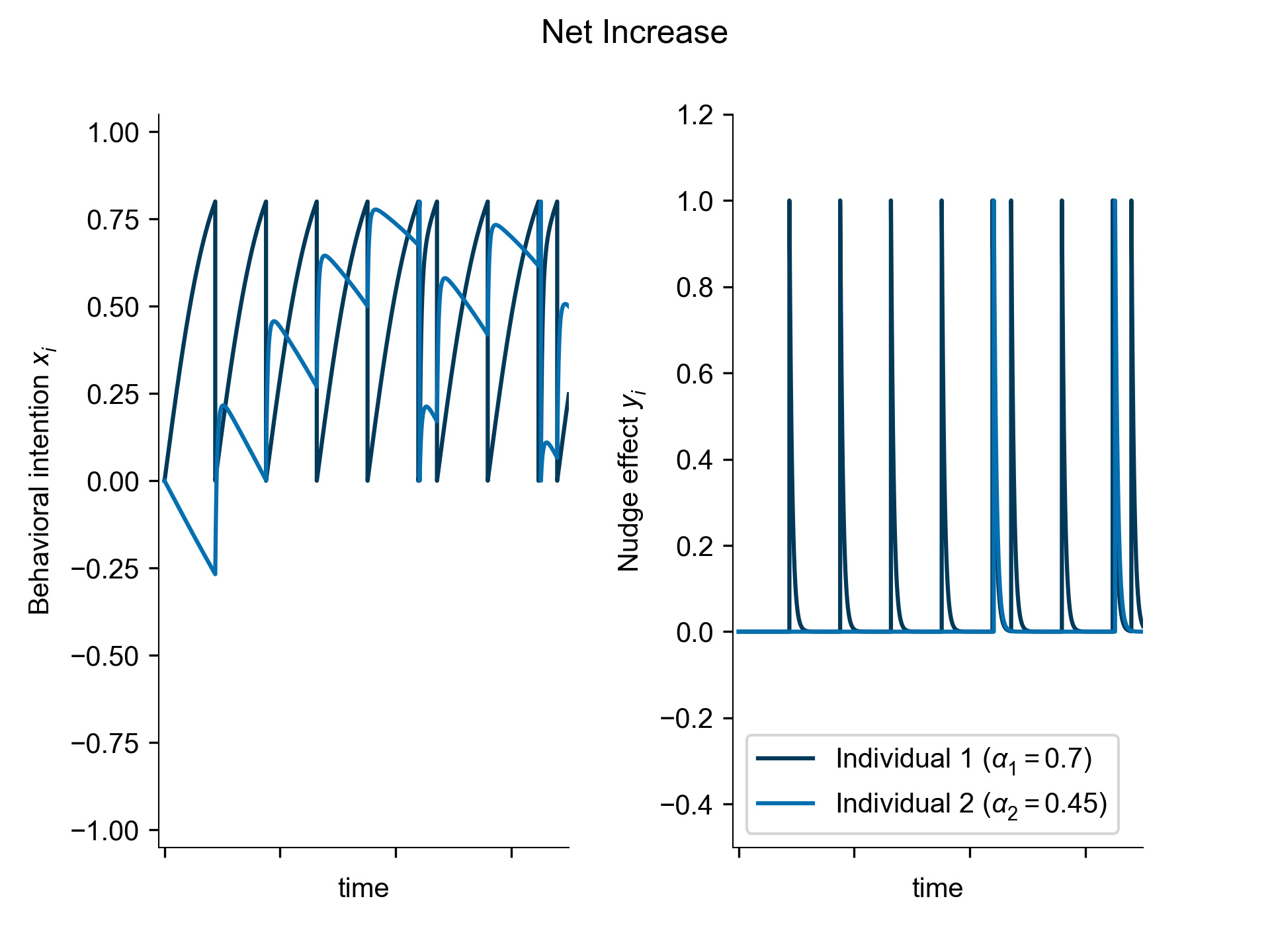}
    \caption{With each action of individual 1, individual 2's behavioral intention initially increases and, although it decreases, it does not decrease far enough by the next time individual 1 acts. This results in a net increase for individual 2's behavioral intention and thus individual 2 will act. Knowing that individual 2 acted once is sufficient for individual 2 to continue to act indefinitely.}
    \label{fig:netIncrease}
\end{figure}

We are able to now summarize our results on the no action, partial action, and full action states into Table 2, in which without loss of generality, we still assume $\alpha_1 \geq \alpha_2$.
\begin{table}[h]
    \begin{center}
    \begin{tabular}{|c|c|c|c|}
    \hline
       $\text{sgn}(\sigma_A\alpha_1 - \sigma_S\mu_S)$  & $\text{sgn}(\sigma_A\alpha_2 - \sigma_S\mu_S)$ & $\text{sgn}(M)$ & Classification \\   \hline
       Negative & (Forced Negative) & &  No Action \\   \hline
       0  & (Forced Non-Positive) &  &  No Action\\   \hline
       Positive  & Positive & &  Full Action \\   \hline
        Positive & 0 &  & Full Action \\   \hline
        Positive & Negative & Positive &  Full Action \\ \hline
        Positive & Negative & Non-Positive &  Partial Action \\ \hline
    \end{tabular}
    \caption{Summary of Classification}
    \label{tab:Classification}
    \end{center}
\end{table}

\section{Remarks on the Difficulty of the System}
The majority of this paper was devoted to the case where $\dot x_1 > 0$ and $\dot x_2 \leq 0$ and $x_2$ had not yet acted. In doing so, we obtained a classification of the action of $x_2$ based on the sign of $M$. Now, we ask a natural follow-up question under those same conditions. To see why this is indeed natural, we'll consider a few examples. 

When $\dot x_2(0) = 0$, it does not matter what $\dot x_1 > 0$ is. $x_1$ will act, and any nudge effect results in $\dot x_2$ experiencing a net increase over a period, thus $x_2$ eventually acts. 

In the exact opposite case, we consider when $\sigma_A\alpha_2 - \mu_S \sigma_S \leq \sigma_S$. Since $ \max \gamma_2 = 1$, we'd then have that $[\sigma_A\alpha_2 +  \sigma_S(\gamma_i - \mu_S)] \leq 0$. Then, it again does not matter what $\dot x_1 > 0$ is. Regardless of when or how fast $x_1$ acts, $x_2$ will never act as it is always non-increasing.

In between these extremes, we can then ask our (hopefully now natural) question: Given that all other parameters are fixed, can we find an $\alpha_1$ such that $x_2$ acts? If so, what is the smallest such value? 

In other words, how small a stimulus can $x_1$ provide such that $x_2$ acts? 

This is intractable. To see why, the first step would be to solve for $T$ satisfying $M > 0$, or \[
B +C  < A_2(2T) + Be^{-r(2T)} + Ce^{-2r(2T)}
\]
But as $T$ is both linear and exponential in two different forms, no elementary function solution exists. If one did, we could easily write $T = \frac{2}{A_1}\tanh(\tau)$ from an earlier proposition and solve for $\alpha_1$ inside of $A_1$, though it would be messy. Alas, this is not possible. 

This hints that many interesting questions about this system are not just difficult, but analytically intractable. As another example: Can we show that $x_2$'s slowest action is its first? This is intuitively clear, \textit{provable} if we suppose that $x_1$ ``ignores'' $x_2$'s nudge effect, which should be a strictly ``worse'' setting, but is tricky at best under our normal conditions. 

Nevertheless, we can handle one non-trivial special case of our first question above. To repeat it: Given that all other parameters are fixed, can we find an $\alpha_1$ such that $x_2$ acts? If so, what is the smallest such value? 

As well as serving as a partial result, this tells us that many attempts to turn intractable questions into tractable ones still leave us with answers that aren't very illustrative.

This special case is when $B = 0$ (alternatively, one could compute also for $C=0$). Recall that $B=-\frac{1}{r}[(\sigma_A\alpha_2-\sigma_S\mu_S)\sigma_C+\sigma_S\sigma_C\mu_C]$, so this is true when \[
\alpha_2 = \frac{\sigma_S(\mu_S - \mu_C)}{\sigma_A}\]
Given certain parameters, to achieve $B = 0$ we may ask for the impossibility that $|\alpha_2| > 1$, so ignore these cases. There certainly are combinations of parameters for which the expression for $\alpha_2$ is valid though. Furthermore, there are combinations of the above parameters that allow $\alpha_2$ to take every value in $[-1,1]$.
\begin{prop}
    Suppose $(h_1)-(h_4)$. Say we have $\alpha_2 = \frac{\sigma_S(\mu_S - \mu_C)}{\sigma_A} \in [-1,1]$, then the minimum $\alpha_1$ such that $x_2$ acts is \[
\alpha_1 
= 
\frac
{
  \tanh^{-1}(\tau)
  +
  \mu_S\sigma_S
  \lt(\frac{C}{2A} + \frac{1}{4r}W\!\lt(-\frac{2rC}{A}e^{-\frac{2rC}{A}}\rt)\rt)
}
{
  \sigma_A
  \lt(\frac{C}{2A} + \frac{1}{4r}W\!\lt(-\frac{2rC}{A}e^{-\frac{2rC}{A}}\rt)\rt)
}
\]
where $W$ is the Lambert $W$ function.
\end{prop}
To see the full computation, refer to \ref{app:lambert}.

The Lambert $W$ function is defined as the inverse to $xe^x$. It is needless to say that this expression is highly problematic to work with for a multitude of reasons. Mainly, these are the sheer size and the required use of Lambert $W$. 

We could have also take then case where $C = 0$ for the same reason why $B = 0$ is acceptable, but this gives a similarly opaque expression.

As mentioned earlier, in the expression $B +C  < A_2(2T) + Be^{-r(2T)} + Ce^{-2r(2T)}$, we cannot solve for $T$. In the special cases where we can, so where $B$ or $C$ is 0, the expression we could obtain would take a difficult form as evidenced above. This means that we, in general, cannot derive a nice expression for the amount of time it takes $x_2$ to act. 

\section{Discussion}\label{sec:Discussion}
The main result here is a full characterization of the motivational dynamics taking place between a pair of individuals with different propensities to act. The classification of behavior can be done entirely from understanding global parameters (i.e. those parameters that describe the relationship of the two individuals) and from understanding the individual characteristics of the individuals (i.e. propensity to act) under the assumptions of the model presented in \cite{schwarze2024planned}. This classification of the behavior in the two individual game can be described as a partition of the parameter space, described bu the critical case of the invariant described in subsection \ref{subsec:Invariant}.  In the two individual case $\sigma_A\alpha_i-\sigma_S\mu_s<0$ for both $i=1$ and $i=2$ then we know we are in no action state trivially. Therefore the region
\[\{(\alpha_1,\alpha_2)\in [0,1]^2; \sigma_A\alpha_i-\sigma_S\mu_s= 0 \text{ and }\sigma_A\alpha_j-\sigma_S\mu_s\leq 0 \text{ for }i\neq j\}\]
Is the boundary between the no action state and the rest of the parameter space. This boundary lies between the no action state and the partial action state except for at the corner ( $\sigma_A\alpha_j-\sigma_S\mu_s= 0 \text{ and }\sigma_A\alpha_2-\sigma_S\mu_s= 0 $) where the no action, partial action, and full action regions all meet.

The boundary between the partial action state and the full action state is given by the invariant $M$. Note that, although $M$ depends on the period $T$ we have shown that if the player with the lower $\alpha$ value acts once, it will surely act again, therefore it is sufficient to use the time to first action by the player with the higher $\alpha$ value as $T$. The time to first action will the be the period for that player until the other player acts for the first time. The time to first action, although it may be difficult to compute, it determined only the the global parameters and by $\alpha_i$ so the can use the sublevel set, super level set, and level sets of $M=0$ to partition the parameter space analytically.

The region
$\{(\alpha_1,\alpha_2)\in [0,1]^2; M=0\}$ describes the boundary between the partial action state and the full action state. Notice that the point on this boundary where $\alpha_1=\alpha_2$ is also on the boundary between the no action state and partial action state. This creates the critical point in the parameter space where all three regions meet.

As a way of demonstrating the power of this result, we also present the result of some numerical experiments carried out by scanning the parameter space of $(\alpha_1,...,\alpha_n)\in [0,1]^n$ for several points in the global parameter space (i.e. several selections of the parameter $\sigma_s, \sigma_c, r, \mu_c,$ and $\tau$. First, in the case, that $n=2$, the entire scan can be visualized on the unit square, and each simulation solution to the system can be classified as ``total action", ``partial action," or ``no action" (Fig \ref{fig:twoindividualpartition}).

\begin{figure}[h]
    \centering
    \includegraphics[width=\linewidth]{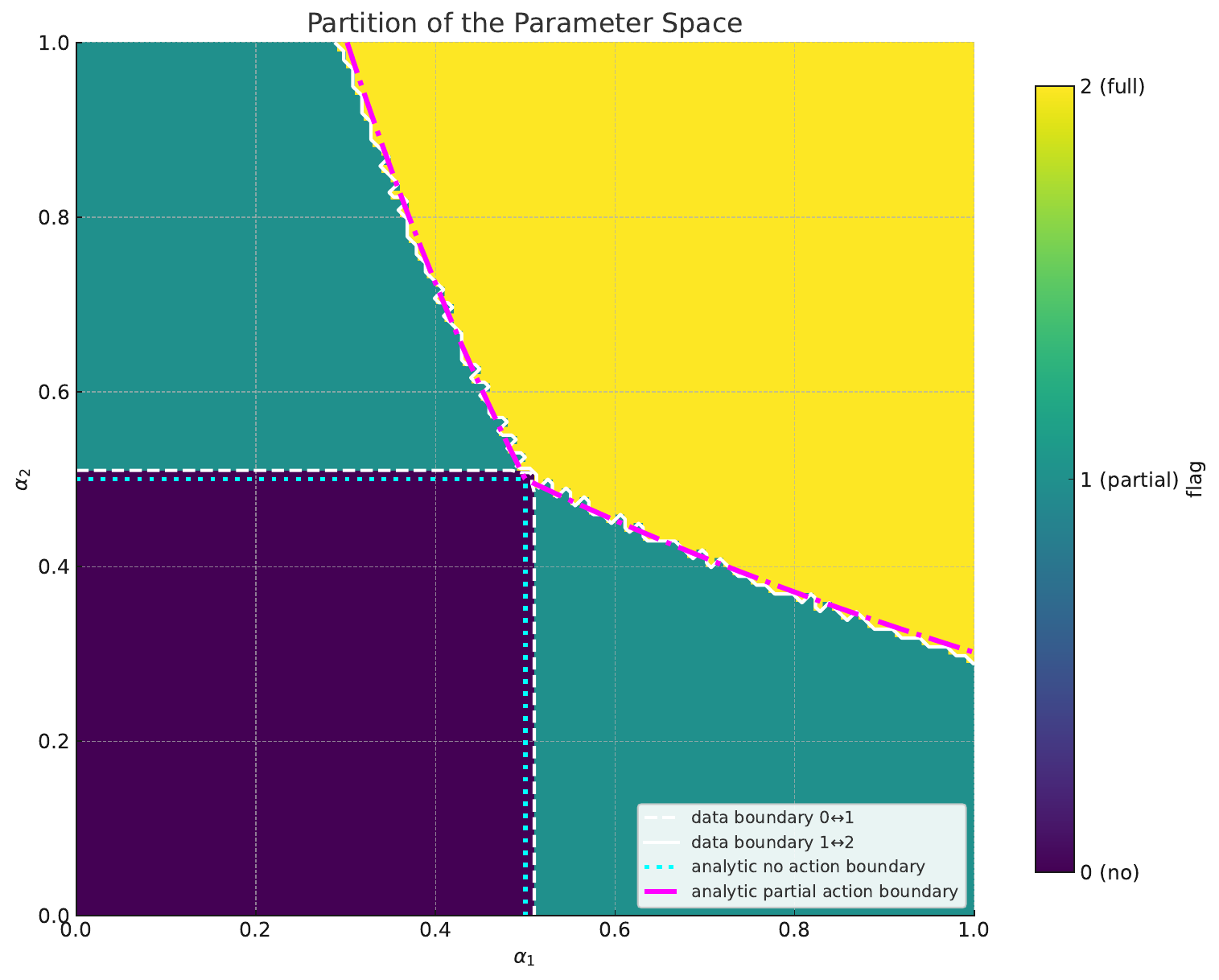}
    \caption{Classification of solutions observed in the parameter space $(\alpha_1,\alpha_2)\in [0,1]^2$ with $\sigma_s= 0.5, \sigma_C= 0.5,r=0.86,\mu_c=0.05$ and $\tau = 0.8$. In dotted lines, the analytical boundaries between regions are shown to nearly perfectly overlap the observed boundaries between qualitatively different behaviors from the simulation.}
    \label{fig:twoindividualpartition}
\end{figure}

\begin{figure}
    \centering
    \includegraphics[width=\linewidth]{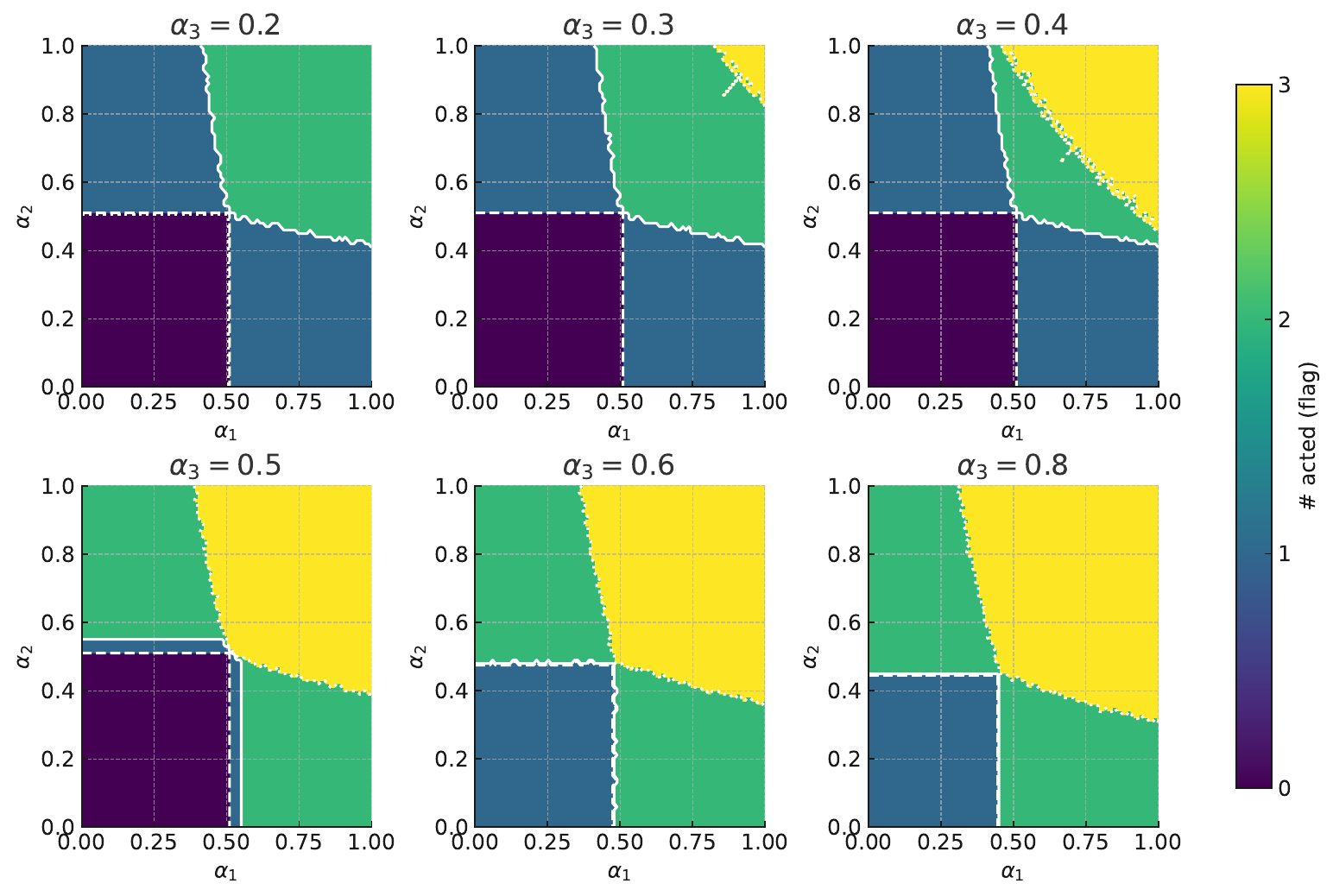}
    \caption{From simulation, behavior of the three player system can be categorized according to the position of the system in the parameter space. For different choices of $\alpha_3$, the action of the system is classified by the number of players acting for each choice of $(\alpha_1,\alpha_2)\in [0,1]^2$.}
    \label{fig:threeIndividualPartition}
\end{figure}

This result is more than an interesting investigation into a model, it is a meaningful step forward in how we understand the theory of planned behavior.   The results of this paper point to the existence of a continuous partition of the parameter space which separate the system into qualitatively different regimes. For the two player case that can be shown with certainty and the partition can be found analytically. Although we do not prove it here, we have given numerical support for the existence of such a partition in the multiplayer case.

The key insights that allowed this step forward are two fold: The first is that in the isolated two individual systems, if a individual acts once they will act again under the assumptions of the model. The other insight is the existence of the invariant descriptor of the change in internal attitude when a player's partner acts. These insights, not only allow us to get a better grip on a hybrid system whose behavior is difficult to describe, but also provide a step towards relating, mathematically, threshold driven action and emergent properties of collective action. Being able to partition the parameter space as we have analytically in the 2-player case and numerically in the 3 player case allows us to simplify the system into several behavioral regimes.

Furthermore, because the two player case is foundational to any multiplayer interaction. The classification of the two player interaction, we believe, can be used as a stepping stone for further classification. Linking the ``no action," ``partial action," and ``Full action" to different kinds of ``dampening," ``neutral," or ``excititory" associations in the full multiplayer game will allow us to leverage the analytical classification we have here to potentially build a deeper understanding of emergent properties of collective behavior, and in particular, the ways that community structure can be a determining factor in the relationship between collective action, threshold-driven action, and the underlying psychological state.

\bibliographystyle{elsarticle-num}
\bibliography{TPBModelExtension}

\appendix
\appendixpage

\section{Derivation of Critical Point in Interval}\label{app:criticalpoint}

Consider the $n=2$ case. Suppose we have that $\sigma_A \alpha_1 - \sigma_S \mu_S > 0$ while $\sigma_A \alpha_2 - \sigma_S \mu_S < 0$. Then, in the case where $\dot x_2(T) > 0$, we can derive a critical time $t_{crit}$ where for each $k$ before $x_2$ acts (if it does at all), $x_2(kT + t_{crit})$ achieves a local maximum on $[kT, (k+1)T]$.

Notice in the ODE for the position of individual $i$ \[
\dot x_i = [\sigma_A \alpha_i + \sigma_S(\gamma_i - \mu_S)]\sigma_C (\gamma_i + \mu_C)(1-x_i)(1+x_i)
\]
that the only term affecting the sign of $\dot x_i$ is \[
\sigma_A \alpha_i + \sigma_S(\gamma_i - \mu_S)
\]
Since we are in the case with only two individuals, $\gamma_1(t) = e^{-rt}$ when $t < T$. We can look for the critical point of $x_1(t)$ in this interval: \[
0 = \sigma_A \alpha_i + \sigma_Se^{-rt} -\sigma_S \mu_S
\]
\[
\mu_S - \frac{\sigma_A \alpha_i}{\sigma_S} = e^{-rt}
\]
\[
-\frac{1}{r}\ln\lt(\lt|\mu_S - \frac{\sigma_A \alpha_i}{\sigma_S}\rt|\rt) = t_{crit}
\]
Note that $t_{crit} < T$, and as mentioned above, before $x_2$ acts, there is a critical time $t_{crit} + kT \in [kT, (k+1)T]$

\section{Derivation of Position of $x_*$}
We'd like to show that $x_*$ indeed falls in $(-1,0)$. Take the following to be a function of $M$\[
x_*(M) = -\frac{1  - \sqrt{1 - (\tanh M)^2}}{\tanh M}
\]
Its derivative is the following
\[
\frac{\text{sech}^2(M) \lt(\sqrt{1 - \tanh^{2}(M)} - 1\rt)}{\tanh^2(M) \sqrt{1 - \tanh^{2}(M)}}
\]
Analyzing this in pieces:
\[
\frac{\text{sech}^2(M)}{\tanh^2(M)} = \frac{\cosh^2(M)}{\cosh^2(M)\sinh^2(M)} = \text{csch}^2(M)
\]
is positive for $M > 0$, and since $0 < \sqrt{1 - \tanh^2(M)} < 1$, we have \[
\frac{\sqrt{1 - \tanh^{2}(M)} - 1}{ \sqrt{1 - \tanh^{2}(M)}} < 0
\]
Hence, $x_*(M)$ is decreasing for $M > 0$, furthermore it is differentiable everywhere but 0, at which the limit from both sides is 0, and the limit as $M \to \infty$ is $-1$.

\newpage

\section{Alternate Proof of Action with Banach Fixed Point Theorem}

\begin{lemma}
$\wt f(x) : = \begin{cases}
    f(x)  & x \in (-1,1) \\
    \pm 1 & x = \pm1
\end{cases}$ continuously extends $f$ to $[-1,1]$
\end{lemma}
\begin{proof} Using the same notation as earlier, we check the limit of
\[
\lim_{x \to 1^-} \tanh(M+ \tanh^{-1}(x))
\]
Since $\tanh^{-1}(x) \to \infty$ as $x \to 1^-$, we are taking \[
\lim_{u \to \infty} \tanh(u) \to 1
\]
as desired. Thus, we can continuously extend $f(x)$ to $\wt{f(x)}$ by specifying $\wt{f(1)} = 1$, and this agrees with the original ODE: \[
\dot x = [\sigma_A \alpha_i + \sigma_S(\gamma - \mu_S)] \sigma_C(\gamma + \mu_C)(1-x)(1+x)
\]
as $x = 1 \implies \dot x = 0$, so this is a fixed point. And picking any other point $x^*$, we have \[
\frac{\wt{f(1)} - \wt{f(x^*)}}{1 - x^*} = \frac{1 - f(x^*)}{1- x^*} < 1
\]
since $\wt{f(x^*)} = f(x^*) >x^*$. \end{proof}
Intuitively, this should be fairly clear since $\dot x_2 = 0$ at both $1$ and $-1$.
\begin{lemma}
$\wt f([-1,1]) = [-1,1]$
\end{lemma}
\begin{proof}
We consider $\wt f$ on $[u_*, 1]$. $\wt f(1) = 1$. For any other point $x$, we know $\wt f(x) > x$. Suppose for contradiction that $\wt f(x') > 1$ for $x' < 1$, since $\wt f$ is differentiable hence continuous, there is a point $u$ where $\wt f(u) = 1$, but then $\wt f'(u) = 0$, so this is not possible. \end{proof}

\begin{theorem}
Given any $x_2(0)$ and $(h_1)$-$(h_4)$, $M>0$, then $x_2$ must eventually act.
\end{theorem}
\begin{proof} Now, since $\wt f$ is a contraction, we are able to apply the Banach fixed point theorem on $[x_* + \eps, 1]$, and since $1$ is the unique fixed point greater than all $\tau$, we know that $x_2$ will eventually act. If $x_2(T,x_{init})$ is less than or equal to $x_*$, we know it will eventually cross $x_*$ with upper bound on the time given by Lemma 2. \end{proof}

\section{Computing a Minimal $\alpha$ for Action, $B = 0$}\label{app:lambert}

Here is what we're left with after setting $B = 0$:
\[
2AT + C e^{-4rT} = C
\]
Our goal is to use the Lambert $W$ function to obtain a solution for $T$. First, we isolate the exponential term:
\[
e^{-4rT} = 1 - \frac{2A}{C}T
\]
and to eventually apply the Lambert $W$ function, we're going to rewrite our expression in terms of $U$ below, and not $T$. So let
\[
U = 1 - \frac{2A}{C}T
\]
Solving for $T$ in terms of $U$, we obtain
\[
T = \frac{C}{2A}(1 - U)
\]
Substituting $T$ back into the exponent, we get
\[
e^{-4rT}
= e^{-4r\lt(\frac{C}{2A}(1 - U)\rt)}
= e^{-\frac{2rC}{A}}\,e^{\frac{2rC}{A}U}
\]
So, we can rewrite $U$ as the following:
\[
U = e^{-\frac{2rC}{A}}\,e^{\frac{2rC}{A}U}
\]
Multiply both sides by the exponential factor
\[
U\,e^{-\frac{2rC}{A}U} = e^{-\frac{2rC}{A}}
\]
Scale and define $Z$ for the Lambert $W$ form
\[
-\frac{2rC}{A}\,U\,e^{-\frac{2rC}{A}U}
= -\frac{2rC}{A}\,e^{-\frac{2rC}{A}}
\quad\text{ where }Z = -\frac{2rC}{A}U
\]
Recognize the standard form for Lambert $W$
\[
Z e^Z = -\frac{2rC}{A}\,e^{-\frac{2rC}{A}}
\]
Apply the Lambert $W$ function
\[
 Z = W\!\lt(-\frac{2rC}{A}\,e^{-\frac{2rC}{A}}\rt)
\]
Return to $U$
\[
U = -\frac{A}{2rC}\,W\!\lt(-\frac{2rC}{A}\,e^{-\frac{2rC}{A}}\rt)
\]
Finally solve for $T$
\[
T = \frac{C}{2A}(1 - U)
= \frac{C}{2A}
+ \frac{1}{4r}\,W\!\lt(-\frac{2rC}{A}\,e^{-\frac{2rC}{A}}\rt)
\]
Solving for $\alpha_1$ gives
\[
\alpha_1
=
\frac
{
  \tanh^{-1}(\tau)
  +
  \mu_S\sigma_S
  \lt(\frac{C}{2A} + \frac{1}{4r}W\!\lt(-\frac{2rC}{A}e^{-\frac{2rC}{A}}\rt)\rt)
}
{
  \sigma_A
  \lt(\frac{C}{2A} + \frac{1}{4r}W\!\lt(-\frac{2rC}{A}e^{-\frac{2rC}{A}}\rt)\rt)
}
\]

\end{document}